\documentclass[11pt]{amsart}
\usepackage[top=3cm, bottom=2.5cm, left=2cm, right=2cm]{geometry}

\usepackage{amsmath,amssymb,amsthm}
\usepackage[pdftex]{graphicx}
\usepackage{epsfig}
\usepackage[usenames,dvipsnames]{color}
\usepackage{enumerate}
\usepackage{float}

\usepackage{setspace}
\newtheorem{theorem}{Theorem}[section]
\newtheorem{definition}{Definition}[section]
\newtheorem{lemma}{Lemma}[section]

\newtheorem{example}{Example}[section]
\newtheorem{remark}{Remark}[section]

\newcommand{\BR}{\mathbb{R}}
\newcommand{\BC}{\mathbb{C}}

\newcommand{\BZ}{\mathbb{Z}}

\newcommand{\cF}{\mathcal{F}}
\newcommand{\cG}{\mathcal{G}}

\newcommand{\ov}{\overline}

\newcommand{\llangle}{\langle \kern -0.2em \langle}	
\newcommand{\rrangle}{\rangle \kern -0.2em \rangle}
\newcommand{\inner}[1]{\left\langle  #1 \right\rangle }

\newcommand{\iinner}[1]{\llangle  #1 \rrangle }

\newcommand{\pr}{\partial}

\newcommand{\la}{\langle}
\newcommand{\ra}{\rangle}

\newcommand{\B}{\tilde{\mathcal{B}}}

\usepackage{hyperref}

\begin{document}
\title[Generalized Fock space]{Generalized Fock-space and fractional derivatives: Uniqueness of Sampling and Interpolation Sets}

\author[ N. Alpay, P. Cerejeiras, and U. K\"ahler]{ N. Alpay$^\dagger$, P. Cerejeiras$^\ddagger$ and U. K\"ahler$^\ddagger$}
\address{$^\dagger$ Department of Mathematics \newline University of California - Irvine, \newline Irvine, CA 92697, U.S.A.
\newline and  \newline
$^\ddagger$ CIDMA - Center for Research and Development in Mathematics and Applications, \newline Department of Mathematics, University of Aveiro \newline Campus Universit\'ario de Santiago \newline 3810-193 Aveiro,
Portugal. }
%
%\date{}

%%%%%%%%%%%%%

\begin{abstract}
In this paper we introduce a Fock space related to derivatives of Gelfond-Leontiev type, a class of derivatives which includes many classic examples like fractional derivatives or Dunkl operators. For this space we establish a modified Bargmann transform as well as density theorems for sampling and interpolation. These density theorems allow us to establish lattice conditions for the construction of frames arising from integral transforms which are linked by the modified Bargmann transform with the Fock space.
\end{abstract}

\maketitle

\section{Introduction}

One principal problem in modern signal and image processing consists the construction of frames arising from discretization of integral transforms (e.g. wavelet and Gabor systems). This problem has its origins in quantum mechanics and in information theory (see J. von Neumann, D. Gabor) where one aims to represent functions in terms of time-frequency atoms which have a minimal support in the time-frequency plane. For practical applications the continuous transforms are substituted by discrete systems in form of frames which  represent a generalization of biorthogonal bases. Usually, frames are obtained by discretizing the parameter space, which leads to the problem of finding conditions for a lattice in the parameter space to be dense enough to create a frame. In the case of Gabor systems K. Gr\"ochenig and Y. Lyubarskii developed a method that makes it possible to find lattices for the construction of Gabor frames with Hermite functions as window functions by connecting the Gabor system with the standard orthonormal basis in the Fock space via the Bargmann transform (see \cite{GL2007,GL2009}). This reduces the problem of finding lattice constants for the frame parameters to the problem of sets of interpolation and uniqueness of entire functions in the Fock space. This last problem has been studied in detail by K. Seip and co-authors, see~\cite{K92, K292, BSK93, Abreu}. 

This construction leads to the question if one can use the similar methods in other situations. However, many of those situations require a different setting. Examples of such are, for instance, the Calogero-Moser dynamical system of one-dimensional N-body problem of N equal particles with a harmonic potential, or applications in mathematical optics, or more general physical processes which require a memory mechanism in the process (see \cite{DDV2009, H2011}). This type of problems motivate an increasing interest in generalized fractional calculus in the last decades. However, here new difficulties arise: while on the one hand fractional calculus is particularly adequate to handle such problems, on the other hand working with fractional derivatives presents a major disadvantage as most of the classical tools are non-existent, e.g., Leibniz formula, chain rule, or translation invariance, to name just a few.

Thus, having this in mind, this method of Lyubarskii turns out to be much more general than the case of Gabor frames seems to indicate. While the classic setting of the Fock space is linked to the classic derivative and multiplication operators we focus our attention on the case of Gelfond-Leontiev derivatives (for further details, see \cite{K1994}). This type of derivatives includes many important examples as special cases, like fractional derivatives of Caputo or Riemann-Liouville type (the latter via a change of the ground state) or difference-differential operators linked to finite reflection groups, also known as Dunkl operators. While the former is being applied in a variety of areas, like fractional mechanics or grey noise analysis in stochastic processes, the latter appears in the study of Calogero-Sutherland-Moser models for n-particle systems. For this type of operators we are going to construct and study the corresponding Fock space and the connected Bargmann transform. This will allow us to establish the necessary density theorems for sampling and interpolation sequences in these Fock spaces under some additional conditions.  To show the applicability of the Gr\"ochenig-Lyubarskii theory we are going to establish lattice conditions for the existence of frames for the corresponding integral transform like in the classic case of the Gabor transform. 
Furthermore, given the importance of the class of operators under consideration we can see applications in quantum mechanics, stochastic analysis, signal processing or other fields, not only for the discussion of frame construction, but also for representation of coherent states in terms of position and momentum operators.

The paper is organized as follows. In Section 2 we recall the definition of a derivative of Gelfond-Leontiev type and present some important examples. In Section 3 we are going to discuss the corresponding Fock space. In Section 4 we present the Bargmann transform in this setting and the density theorems for sampling and interpolation sequences. We prefer to move the necessary proofs into its own Section 5. In the last Section we will present the application of our density theorems to obtain lattice densities for the construction of frames of the corresponding integral transforms.\\

\section{Preliminaries}

We are going to study Fock spaces related to Gelfond-Leontiev operators. To this end let us start with the definition of operators of generalized differentiation and integration with respect to a given entire function.

\subsection{Generalized fractional derivatives}

 Since we are interested in generalized Fock spaces we begin by consider Gelfond-Leontiev operators with respect to an entire function. Before looking into that we want to remark that one could also consider functions analytic in a disk which would lead to Hardy spaces instead of Fock spaces.

\begin{definition}%[\cite{GL, Kir1}]
Let
\begin{equation} \label{Eq:EntireFunction}
\varphi(z) =\sum_{k=0}^\infty \varphi_k z^k,
\end{equation}
be an entire function with order $\rho >0$ and  degree $\sigma > 0,$ that is, such that $\lim_{k \rightarrow \infty} k^{\frac{1}{\rho}} \sqrt[k]{|\varphi_k|} =\left(\sigma e \rho\right)^{\frac{1}{\rho}}.$ We define the Gelfond-Leontiev (GL) operator of generalized differentiation with respect to $\varphi,$ denoted as $D_\varphi,$
as the operator acting on an analytic function $f(z) =\sum_{k=0}^\infty a_k z^k,  |z|<1,$ as
\begin{equation} \label{Eq:p1+p2}
f(z) =\sum_{k=0}^\infty a_k z^k \quad \mapsto \quad D_\varphi f(z) =\sum_{k=1}^\infty a_k \frac{\varphi_{k-1}}{\varphi_k} ~z^{k-1}. \end{equation}
\end{definition}

Hence, under the condition on $\varphi$ that $\lim \sup_{k \rightarrow \infty} \sqrt[k]{\left|\frac{\varphi_{k-1}}{\varphi_k}\right|} =1$ by the Cauchy-Hadamard formula we have that the series in (\ref{Eq:p1+p2})
inherit the same radius of convergence $R>0$ of the original series $f.$

Also, we like to point out that the function $\varphi$ acts as a replacement of the \textit{exponential function} for the Gelfond-Leontiev operator of generalized differentiation. Indeed, $D_\varphi \varphi =\varphi.$

Let us take a look at some examples. The first example is just the classic derivative.

\begin{example}\label{ex21}
	 For $\varphi(z) = e^z$, with $\varphi_k = 1/ \Gamma(k+1)$ for  $k=0,1, 2, \ldots,$. We get:
	$$D_\varphi f(z) = D_\varphi \left( \sum_{k=0}^\infty a_k z^k \right) = \sum_{k=1}^\infty a_k \,k\,z^{k-1} = \partial_z f(z),$$
	since $ \frac{\varphi_{k-1}}{\varphi_k} = \frac{k!}{(k-1)!}=k$. %\frac{\Gamma(n)}{\Gamma(n-1)} =
\end{example}

The next example is a classic example of a fractional derivative.

\begin{example} \label{exML}
	Let $\varphi$ be the Mittag-Leffler function defined as:
	\begin{equation}\label{Eq:ML_function}
		E_{\frac{1}{\rho},\mu}(z) =\sum_{k=0}^\infty \frac{z^k}{\Gamma\left(\mu +\frac{k}{\rho}\right)}, \qquad \rho>0, ~\mu \in \BC, ~{\rm Re}(\mu)>0,
	\end{equation}
	with coefficients $\varphi_k  =\frac{1}{\Gamma\left(\mu +\frac{k}{\rho}\right)}.$ The operator (\ref{Eq:p1+p2}) becomes then the Dzrbashjan-Gelfond-Leontiev operator:
	$$D_{\rho,\mu} f(z) =\sum_{k=1}^\infty a_k \frac{\Gamma\left(\mu +\frac{k}{\rho}\right)}{\Gamma\left(\mu +\frac{k-1}{\rho}\right)} ~z^{k-1}.$$
\end{example}

That the range of possibilities for fractional derivatives is much larger can be seen in the next example.

\begin{example}
	For 	$\varphi_k=\frac{b a^{\frac{k+1}{b}}}{\Gamma \left( \frac{k+1}{b} \right)}$ with ${\rm Re}(a)>0, b>0,$ and $k=0,1,2, \ldots,$ we have
	\[
	D_\varphi f(z) = \sum_{k=1}^\infty a_k \left[ \frac{b a^{\frac{k}{b}} \Gamma \left( \frac{k+1}{b} \right)}{\Gamma \left( \frac{k}{b} \right) b a^{\frac{k+1}{b}}} \right]z^{k-1}
	= \sum_{k=1}^\infty a_k \frac{\Gamma\left(\frac{k+1}{b}\right) }{ a^{\frac{1}{b} } \Gamma\left(\frac{k}{b}\right) }  z^{k-1}.
	\]
	
	We remark that for $a=b=1$ we have the usual case $\varphi(z) = e^z$. 
	\end{example}

Although in this work we consider entire function, we can also give these additional examples with functions that are not entire.

\begin{example}\label{exBS}
	For $\varphi(z) = \frac{1}{1-z}, ~|z| <1,$  we have $\varphi_k=1$ and so
	\[ D_\varphi f(z) = \left( \sum_{k=0}^\infty a_k z^k \right) = \sum_{k=1}^\infty a_k  z^{k-1} = \frac{f(z)-f(0)}{z}.\]
	
	In this case, the GL operator is also known as \emph{backward-shift} operator.
\end{example}

\begin{example}\label{exELN}
	Consider $\varphi(z) = \sum_{k=0}^\infty  
		\frac{z^k}{\Gamma^{(n)}(k+1)}$, for a fixed $n \in \mathbb{N}.$
	We get
	\[ D_\varphi f(z) = \sum_{k=1}^\infty a_k \left[ \frac{
	\Gamma^{(n)}(k+1)}{
	\Gamma^{(n)}(k)}  \right] z^{k-1}.  \]
		
	For $n=1$ we have $ D_\varphi f(z) = \sum_{k=1}^\infty k a_k  c_{k}  z^{k-1},$ where $c_k$ is a constant depending on the Euler–Mascheroni constant. This example will be re-taken later in the context of creation and annihilation operators related to a fractional Fock space, see example~\ref{ex:bt2}. \end{example}

Those are a few examples of fractional GL-class of derivatives. This class also includes the Caputo and Riemann-Liouville derivatives, the latter being obtained by considering the GL derivative as an operator acting on the ground state $z^{1-\delta}$. But the class of GL derivatives is much broader. We present now an example of a derivative whose connection with GL derivatives may not be so well-known.

\subsection{Dunkl operators}\label{sec:dunkl} Another important example of a generalized differentiation operator of Gelfond-Leontiev type is the case of Dunkl operators, also called differential-difference operators linked to a finite reflection group (see \cite{Roesler}).

These operators are introduced as follows. Given a non-zero vector $\nu \in \mathbb{R}^n$ let $\sigma_\nu (x)$ denote the reflection of a given vector $x\in\BR^n$  on the hyperplane orthogonal to $\nu$. A root system $R$ is a finite set of non-zero vectors in $\mathbb{R}^n$ such that $\sigma_\nu R=R$ e $R\cap \mathbb{R}\nu=\{\pm  \nu \}$ for all $\nu \in R$.  A positive subsystem $R_+$ is any subset of $R$ satisfying $R=R_+\cup(-R_+).$ This implies that $R_+$ and $-R_+$
are separated by a hyperplane passing through the origin.

A Coxeter group (or finite reflection group) $\cG$ is a group generated by the reflections $\sigma_\nu, \nu \in R,$ thus, it is a  subgroup of the orthogonal group $O(n)$. Standard examples are the groups $A_{n-1}$
and $B_n$ (see e.g. \cite{Roesler}, \cite{CKR}).  A multiplicity function $\kappa_\nu$ is a $\cG$-invariant complex-valued function defined on $R$, i.e., $\kappa_\nu=\kappa_{g \nu}$ for all $g\in\cG$. For a chosen
positive subsystem $R_+$ we introduce the index
$$\gamma_\kappa=\sum_{\nu\in R_+}\kappa_\nu, $$
and the weight function
$$h_{\kappa}(x)=\Pi_{\nu \in R_+}|<\nu,x>|^{\kappa_\nu},$$
where $<\cdot, \cdot>$ denotes the Euclidean inner product in $\mathbb{R}^n.$

For each fixed positive subsystem $R_+$ and multiplicity function $\kappa_\nu$ we have, as invariant operators, the Dunkl operators (or differential-difference operators):
$$T_j f(x)=\frac{\partial}{\partial x_j}f(x)+\sum_{\nu \in R_+}\kappa_\nu \frac{f(x)-f(\sigma_\nu x)}{<x,\nu>}\nu_j.$$

Associated to these is the \textit{intertwining operator} which allows to interchange Dunkl derivatives with the usual partial derivatives. Let $\Pi$ denote the space of homogeneous polynomials. Furthermore, let
$\Pi_k$ denote the space of homogeneous polynomials of degree $k.$

\begin{lemma}[\cite{Roesler}]
If the multiplicity function $\kappa$ is such that $\cap_{j} \ker T_j = \mathbb{C}$ then it exists a unique positive linear isomorphism $V_{\kappa} : \Pi \to \Pi,$ denoted as intertwining operator, which satisfies
 \begin{enumerate}
 \item $V_{\kappa} (\Pi_k) \subseteq  \Pi_k;$
 \item $V_{\kappa} \left|_{ \Pi_0} \right. = id;$
 \item $T_j V_{\kappa} = V_{\kappa} \partial_j, ~$ with $V_{\kappa} (1)=1.$
 \end{enumerate}
\end{lemma}

This means that we can express the Dunkl operators  in terms of generalized differentiation operators with respect to the function $\varphi(z)=V_\kappa(e^z).$

For instance, in the rank-one case we have
$$
V_\kappa (z^{2n})=\frac{\left(\frac{1}{2}\right)_n}{\left(\kappa +\frac{1}{2}\right)_n}z^{2n}, \qquad V_\kappa (z^{2n+1})=\frac{\left(\frac{1}{2}\right)_{n+1}}{\left(\kappa+\frac{1}{2}\right)_{n+1}}z^{2n+1},
$$
where $(a)_0=1,$ and $(a)_n=\frac{\Gamma(a+n)}{\Gamma(a)}, ~{\rm Re}(a) >0,$ denotes the Pochhammer symbol, or rising factorial.
This leads to the function
$$
\varphi(z)=e^z {}_1 F_1(\kappa,2\kappa+1;-2z),
$$
with
$$
\varphi_{2n}= \frac{\left(\frac{1}{2}\right)_n}{(2n)!\left(\kappa+\frac{1}{2}\right)_n} \qquad\mbox{and}\qquad \varphi_{2n+1}=\frac{\left(\frac{1}{2}\right)_{n+1}}{(2n+1)!\left(\kappa+\frac{1}{2}\right)_{n+1}}.
$$

\section{Fractional Fock space}

The classical Bargmann-Fock space links to quantum mechanics through the Schr\"odinger equation which describes the evolution of the state of the system by means of the Hamiltonian. In this space the momentum $P$ and position $Q$ operators, which describe the observables, are related by canonical commutation relations as well as by duality. The Fock space $\mathcal F$ is defined as the set of all entire functions $f$ such that $\| f \| < \infty,$ whereas the norm is induced by the inner product
\begin{equation}\label{ClassicalFock}
\langle f, g \rangle = \frac{1}{\pi}\int_{\mathbb C} \overline{f(z)} g(z) e^{-|z|^2}dxdy, \quad f, g \in \mathcal F.
\end{equation}
Hence, $\mathcal F$ can be seen as the reproducing kernel Hilbert space with reproducing kernel given by
	\begin{equation}\label{RKHS_kernel}
k(z,w) = e^{{\ov z} w}.
	\end{equation}
	
Furthermore, the Fock space is the unique Hilbert space of entire functions in which the momentum operator coincides with the classic derivative while the position operator is the multiplicative operator. This establishes a framework for other similar characterizations of spaces of analytic functions such as the Hardy space and Dirichlet space which was done in previous work (N. Alpay~\cite{2}).

Here we are interested in the Fock space related to our GL derivative, which we are going to introduce next. Although $\varphi$ is an entire function with complex coefficients, in the following we additionally assume $\varphi$ to have positive coefficients, that is, $\varphi_n >0$ for all $n,$ in order to ensure positivity of the measure. In fact, it will allow us to obtain a probability measure similar to the classic case.

\subsection{Inner product}\label{subSec:IP}
Given two entire functions $f(z) =\sum_{k=0}^\infty f_k z^k, g(z) = \sum_{k=0}^\infty g_k z^k,$  we consider the following Hilbert spaces

\begin{enumerate}[(i)]
\item the fractional space $\ell^2_\varphi$ of the sequences $f \sim (f_k)_{k=0}^\infty$ and weighted inner product
\begin{equation}\label{Eq:InnerProduct_l}
\inner{f, g}_{2, \varphi} = \sum_{k=0}^\infty  \frac{\ov{f_k} g_k}{ \varphi_k};
 \end{equation}
\item the fractional Fock space $\cF_\varphi$ endowed with the weighted inner product \begin{equation}\label{Eq:InnerProduct_Fock}
\iinner{f, g}_{\cF, \varphi} = \frac{1}{\pi} \int_{\BC} \ov{f(z)} g(z) K_\varphi(-|z|^2) dxdy,
 \end{equation}and where $K_\varphi$ denotes the weight function. 
 We remark that in the classical case  $K_\varphi(-|z|^2)$ is the Gaussian $e^{-|z|^2/2}$.
\end{enumerate}

We aim to identify in an isometric way a function $f(z)= \sum_{k=0}^\infty f_k z^k$ in $\cF_\varphi$ with its sequence of coefficients $(f_k)_{k=0}^\infty$ in $\ell^2_\varphi$. This means that these weighted inner products  should be related by the identity $\inner{f, g}_{2, \varphi} = \iinner{f, g}_{\cF, \varphi}$, i.e.

\begin{equation}\label{Eq:Relation}
\inner{f, g}_{2, \varphi} = \sum_{k=0}^\infty  \frac{\ov{f_k} g_k}{ \varphi_k} = \frac{1}{\pi} \int_{\BC} \ov{f(z)} g(z) K_\varphi(-|z|^2) dxdy = \iinner{f, g}_{\cF, \varphi}.
 \end{equation}

%From the Fock space $\cF_\varphi$ we have f
For $f(z) = z^k$ and $g(z)=z^n$ we have
\begin{eqnarray*}
\frac{\delta_{n,k} }{\varphi_n }  = \iinner{z^k, z^n}_{\cF, \varphi}, \nonumber
& = & \frac{1}{\pi} \int_{\BC} \ov{z}^k z^n K_\varphi(-|z|^2) dxdy, \nonumber \\
& = & \frac{1}{\pi} \int_0^\infty  \int_0^{2\pi} r^{k+n} e^{i (n-k) \theta} K_\varphi(-r^2)  d\theta  r dr , \nonumber \\
& = & 2 \delta_{n,k} \int_0^\infty  r^{k+n+1}  K_\varphi(-r^2)  dr.
\end{eqnarray*}
This leads us to
\begin{eqnarray}
\frac{1 }{\varphi_n }  & = & 2 \int_0^\infty  r^{2n+1}  K_\varphi(-r^2)  dr, \nonumber \\
& = &  \int_0^\infty  x^n  K_\varphi(-x)  dx, \quad (x = r^2), \nonumber \\
& = & \mathcal{M}(\tilde K_\varphi) (n+1), \label{Eq: Mellin transform}
\end{eqnarray} where $\mathcal{M}$ denotes the Mellin transform of the weight function $\tilde K_\varphi(x) := K_\varphi(-x)$ evaluated at the point $n+1.$ This reduces the determination of the measure $K_\varphi (-|z|^2) dx dy$ either to an inversion of the Mellin transform or to a Stieltjes moment problem. 
Of course, a sufficient condition for the determination of $K_\varphi$ consists in the Carleman condition
$$
\sum_{n=1}^\infty\frac{1}{\varphi_n^{\frac{1}{2n}}}=+\infty.
$$

Let us first consider the classic case as an example:
\begin{example}\label{exponential} Again, for $\varphi(z) = e^z,$ we have
$$\frac{1}{\varphi_n}= n!  =  \int_0^\infty  x^n  K_\varphi(-x)  dx =  \mathcal{M}(\tilde K_\varphi) (n+1). $$
Moreover, as
$$n! = \Gamma(n+1) = \int_0^\infty  x^n  e^{-x} dx$$
we identify the weight function as
$$K_\varphi(x) = e^{x},$$ and leading to the classic inner product in the Fock space
$$\iinner{f, g}_{\cF, \varphi} = \frac{1}{\pi} \int_{\BC} \ov{f(z)} g(z) e^{-|z|^2} dxdy.$$
\end{example}

\begin{example}
	When $K_\varphi(z)=E_{\frac{1}{\rho},\mu}(z),$ that is, the Mittag-Leffler function defined as in \eqref{Eq:ML_function}
	then the weighted inner product would be given by \eqref{Eq:Relation}
	\[
	\iinner{f, g}_{\cF, \varphi} = \frac{1}{\pi} \int_{\BC} \ov{f(z)} g(z) E_{\frac{1}{\rho},\mu}(-|z|^2) dxdy.
	\]
\end{example}

\begin{example}\label{ex33}
	
	Consider $\varphi(z) = \sum_{k=0}^\infty %\frac{\pr^n}{\pr k^n}
	\frac{z^k}{\Gamma^{(n)}(k+1)}$, for some fixed $n \in \mathbb{N}$ with coefficients $\varphi_k=\frac{1}{\Gamma^{(n)}(k+1)}$.
	
	Using the Mellin transform we get the weight function
	$$K_\varphi (-|z|^2)  =2 e^{-|z|^2} \ln^n |z|. $$

\end{example}

\begin{example}\label{ex34}
	
	Consider $\varphi(z) = \sum_{k=0}^\infty \frac{z^k}{\pi\cot(\pi(k+1))}$, for some fixed $n \in \mathbb{N}$ with coefficients $\varphi_k=\frac{1}{\pi \cot (\pi(k+1) )}$.
	
	Using the Mellin transform we get the weight function
	$$K_\varphi (-|z|^2)  = \frac{1}{1+|z|^2}.$$%2 e^{-|z|^2} \ln^n |z|. $$
	
\end{example}

We can now discuss the multiplication and derivative operators in this Fock space $\cF_\varphi.$ %are also the position and momentum operators.
As usual the multiplication operator is given by
\begin{equation}\label{Eq:Position}
M_z f (z) := z f(z) = \sum_{k=0}^\infty  f_k z^{k+1} = f_0 z + f_1 z^2 + f_2 z^3 + \cdots
 \end{equation}
defined over the domain $ \mathrm{Dom}(M_z)=\{F\in \mathcal F: zF\in\mathcal{F}\}$.  One can observe that $M_z$ induces a shift in $\ell^2_\varphi.$

Its dual $M^\ast_z$ is defined by
$$
\iinner{M_z f,g}_{\cF, \varphi} = \iinner{ f, M_z^\ast g}_{\cF, \varphi}.
$$
This can be easily calculated by passing to the space $\ell_\varphi^2$:
\begin{gather}
\iinner{M_z f,g}_{\cF, \varphi} =\inner{M_z f,g}_{2, \varphi} = \sum_{k=0}^\infty  \ov{f}_k ~g_{k+1} \frac{1}{ \varphi_{k+1}} = \sum_{k=0}^\infty  \ov{f}_k ~ \left( g_{k+1} \frac{\varphi_{k}}{\varphi_{k+1}}  \right) \frac{1}{\varphi_{k}} = \inner{
f, M_z^\ast g}_{2, \varphi}= \iinner{ f, M_z^\ast g}_{\cF, \varphi}.
\end{gather}
Therefore, the dual $M_z^\ast$ is defined over the domain $ \mathrm{Dom}(M_z^\ast)=\{F\in \mathcal F: D_{\varphi} F\in\mathcal{F}\}$ by
\begin{equation}\label{Eq:Derivative1}
M_z^\ast g (z) := \sum_{k=0}^\infty  g_{k+1}~ \frac{\varphi_{k}}{\varphi_{k+1}} z^k = g_1 \frac{\varphi_{0}}{\varphi_{1}} + g_2 \frac{\varphi_{1}}{\varphi_{2}}  z^2 + g_3 \frac{\varphi_{2}}{\varphi_{3}} z^3 + \cdots
= D_{\varphi} g(z).
 \end{equation} It is also an easy task to prove that this Fock space is the unique space where the associated GL-derivative is the dual operator to the multiplication operator.

We also need to point out that an orthonormal basis $\{e_n\}$ for our Fock space $\cF_\varphi$ is given by $e_n(z)=\sqrt{\varphi_n}z^n$. By looking at the action of $M_z$ and $D_{\varphi}$ on the orthonormal basis it  follows that $\mathrm{Dom}(M_z^\ast)=\mathrm{Dom}(M_z)$.

We are now going to study the reproducing kernel property of our Fock space. We recall that the weight $K_\varphi(-|z|^2)$ has to satisfy the property for $K_\varphi$
$$	\frac{1}{\varphi_n} = \mathcal{M}(K_\varphi(- \cdot))(n+1),$$
whereas $ \mathcal{M}$ denotes the Mellin transform. This relation also induces a discrete reproducing kernel given by
\begin{equation}
	\label{eq: k_varphi}
	k_\varphi(n, k) := \varphi_n \delta_{n,k}, \qquad n, k \in \mathbb N_0,
	\end{equation}	
and we define the corresponding discrete reproducing kernel Hilbert space as 
\begin{equation}
	\label{eq: H(k_varphi)}
\mathcal H(k_\varphi) := \Big\{ \underline f := (f_n)_{n=0}^\infty :  \| \underline f \|_{\ell^2_\varphi}^2  = \sum_{n=0}^\infty \frac{|f_n|^2}{\varphi_n} < \infty  \Big\}.	
\end{equation}

For all sequences $\underline f\in \mathcal H(k_\varphi)$ we have
\begin{equation}
	\label{eq:f_in_H}
	f_n = \la k_\varphi(n, \cdot ), \underline f \ra_{2, \varphi}, \quad n \in \mathbb N_0.
\end{equation}
From the Cauchy Schwarz inequality $|\la f,g\ra_{2, \varphi} | \leq \| f \|_{\ell^2_\varphi} \|g \|_{\ell^2_\varphi}$ and $f(z)=\la k_\varphi(z, \cdot), f \ra_{2, \varphi}$, we have
$$|f(z)|\leq \| f \|_{\ell^2_\varphi}  \| k_\varphi(z, \cdot) \|_{\ell^2_\varphi}.$$

Let us remark that the continuous kernel $K_\varphi(z, w) := \varphi(\ov z w) = \sum_{n=0}^\infty \varphi_n (\ov z w)^n,$ will be a reproducing kernel in the Hilbert space 
\begin{equation}
	\label{eq: H(K_varphi)}
\cF_{\varphi} := \Big\{ f := \sum_{n=0}^\infty f_n z^n :  \iinner{f, f}_{\cF, \varphi}  < \infty  \Big\}.	
\end{equation}

Recall here that $\varphi$ as in (\ref{Eq:EntireFunction}) is an entire function with order $\rho >0$ and degree $\sigma > 0.$  

We look now into the continuous kernel associated to $k_\varphi.$ Combining $\inner{f, g}_{2, \varphi} = \iinner{f, g}_{\cF, \varphi}$ with the reproducing kernel property (\ref{eq:f_in_H}) we obtain
\begin{eqnarray}	\label{eq:ContinuousKernel}
	f(z)  & = & \la k_\varphi(z, \cdot), f \ra_{2, \varphi} \nonumber \\
	& = &   \frac{1}{\pi} \int_{\BC} \ov{k_\varphi(z, w)} f(w) K_\varphi(-|w|^2) dxdy \qquad (w=x+iy) \nonumber \\
	& = &  \frac{1}{\pi} \int_{\BC} \underbrace{\ov{\varphi(\ov z w)} }_{=:} f(w)  K_\varphi(-|w|^2) dxdy \nonumber \\
	& = &  \frac{1}{\pi} \int_{\BC} \ov{\mathbb{K}_\varphi(z, w)} f(w) K_\varphi(-|w|^2) dxdy,
\end{eqnarray} where  $\mathbb{K}_\varphi(z, w) := \varphi(\ov z w) $ denotes the continuous reproducing kernel with respect to the weighted measure $d\mu(w) = K_\varphi(-|w|^2) dxdy$.

\begin{example}\label{exponential2} For $\varphi(z) = e^z,$ we obtain
$$\mathbb{K}_\varphi(z, w) = e^{\ov z w} = \varphi(\ov z w),$$
with the Gaussian weighted measure $d\mu(z) = e^{-|z|^2} dxdy.$
\end{example}

Using our reproducing kernel we have the following characterization of bounded operators on $\mathcal{F}$
\begin{theorem}\label{th:31}
Let $T$ be a bounded operator on $\mathcal F$ and $\mathbb{K}_{\varphi,T}(\overline{z}, w)=T^\ast(\mathbb{K}_\varphi(\overline{z},\cdot))(w)$. Then $\mathbb{K}_{\varphi,T}$ has the following properties
\begin{enumerate}
  \item $\mathbb{K}_{\varphi,T}$ is an entire function on $\mathcal{C}^2$.
  \item $\mathbb{K}_{\varphi,T}(\cdot, w)\in\mathcal{F}$ for all $w$ and $\mathbb{K}_{\varphi,T}(z, \cdot)\in \mathcal{F}$ for all $z$.
  \item $|\mathbb{K}_{\varphi,T}(\overline{z},w)|\leq K_\varphi(|z|^2)K_\varphi(|w|^2)\|T\|$.
  \item $TF(z)=\int_{\mathbb{R}^2}\overline{\mathbb{K}_{\varphi,T}F(w)}K_\varphi(-|w|^2)dxdy$ for all $F\in\mathcal{F}$ and $z\in\mathcal{C}$.
  \end{enumerate}
\end{theorem}

The proof is a straightforward adaptation of the proof of Proposition (1.68) in~\cite{Folland}. In particular, this means that any bounded operator is determined by the action of its adjoint on the reproducing kernel $ T^\ast_w\varphi(\overline{z}w)$

\subsection{Generalized Bargmann transform}\label{sec:B}

One of the important links of the Fock space to applications is given by the Bargmann transform which allows to transform problems over the space $L^2(\mathbb{R})$ into problems over the Fock spaces which is also closely linked to the Bargmann-Fock representations of the Weyl-Heisenberg group.

It is well known that the system of Hermite functions $h_n:\mathbb{R}\to\mathbb{R}$ given by
$$
h_n(x)=\frac{1}{\pi^{1/4}2^{n/2}\sqrt{n!}}H_n(x)e^{-\frac{x^2}{2}},\quad n \in \mathbb{N}_0
$$
where $H_n$ denote the Hermite polynomials, forms an orthonormal basis in $L^2(\mathbb{R})$. If we map each $h_n = h_n(x)$ into $\frac{z^n}{\sqrt{n!}}$ we get the so-called Bargmann transform:
$$
\mathcal{B} : L^2(\mathbb{R}) \to \mathcal{F}
$$
given by
\[ \mathcal{B} f (z) = \int_\mathbb{R}  \left( \sum_{n=0}^\infty \overline{ h_n(x) }\frac{\pi^{n/2} z^n}{\sqrt{n!}} \right) f(x) dx, %= \sum_{n=0}^\infty f_n \frac{z^n}{\sqrt{n!}}  },
\quad f \in L^2(\mathbb{R}). %, \blue{f_n = \inner{h_n, f}}
\]
The calculation of the sum of this series gives the well-known formula for the kernel of the Bargmann transform $k(x,z)=2^{1/4}e^{2\pi xz-\pi x^2 -(\pi/2)z^2}$ so that in closed form the Bargmann transform appears as a double version of the Weierstra\ss  ~transform. It also was as an obvious consequence that the Bargmann transform is a unitary isomorphic mapping between $L^2(\mathbb{R})$ and the Fock space $\mathcal{F}$.

This mapping provides a large number of applications including mapping a windowed Fourier transform of a signal in $L^2(\mathbb{R})$ with a window given by a Hermite function into an analytic function belonging to the Fock space (see \cite{GL2007, J2005}). Moreover, it also allows to consider the pre-image of the annihilation and creation operators in the Fock space, given by $\partial_z$ and $M_z,$ as operators over $L^2(\mathbb{R})$ which in the classic case turns out to be the classic position and momentum operators together with their corresponding coherent states.

In our case we consider the modified Bargmann transform which maps $h_n=h_n(x)$ into $\sqrt{\varphi_n} z^n.$ This correspondence allow us to link the multiplication and derivative operators in the fractional Fock space with creation and annihilation operators in the classic $L^2$-space which leads different types of coherent states such as the squeezed coherent states.

Consider the modified Bargmann transform $\tilde{\mathcal{B}}: L^2(\mathbb{R}) \to \mathcal{F}_\varphi$ given by
\begin{align*}
	\tilde{\mathcal{B}}f(z)&=\int_\mathbb{R}  \left( \sum_{n=0}^{\infty} \overline{h_n(x)} \sqrt{\varphi_n}z^n \right)  f(x)  dx
	=\sum_{n=0}^{\infty} f_n \sqrt{\varphi_n}~ z^n,
\end{align*}
with $f_n=\int_\mathbb{R} \overline{ h_n(x)}  f(x)  dx = \inner{h_n, f}_{L^2(\BR)}$.

Since by construction $\tilde{\mathcal{B}}$ is a unitary operator we have $\tilde{\mathcal{B}}^\ast = \tilde{\mathcal{B}}^{-1}$. Hence, the inverse of the Bargman transform, $\tilde{\mathcal{B}}^{-1},$ can be found in the following manner:
\begin{align*}
	\iinner{ \B g, F}_{\cF,{\varphi}} 	&= \int\limits_{\BC}\overline{\Big[ \int\limits_{\BR} \big( \sum_{n=0}^{\infty} \overline{h_n(t)}\sqrt{\varphi_n}z^n  \Big) g(t) dt \Big]} F(z) d\mu(z)\\
	&=\int\limits_{\BR} \overline{g(t)}\sum_{n=0}^{\infty} \Big( \underbrace{\int\limits_{\BC}  \sqrt{\varphi_n}\overline{z}^n F(z) d\mu(z)}_{:=  F_n  = \iinner{\sqrt{\varphi_n} z^n, F}_{\cF,{\varphi}} } \Big) h_n(t) dt \\
	& = \langle g,  \B^{-1}F \rangle_{L^2(\BR)}.
\end{align*}

By considering the action of the multiplication and GL operators, $M_z$ and $%\partial
D_\varphi,$ on $\tilde{\mathcal{B}}f$ we get
\[ M_z\tilde{\mathcal{B}}f(z) = z\tilde{\mathcal{B}}f(z)=\sum_{n=0}^{\infty}  z^{n+1} f_n\sqrt{\varphi_n} =\sum_{n=1}^{\infty}  z^{n} f_{n-1}\sqrt{\varphi_{n-1}} = \sum_{n=1}^{\infty}  z^{n} f_{n-1}\sqrt{\varphi_{n}} \sqrt{\frac{\varphi_{n-1}}{\varphi_{n}}},\]
as well as
\[ %\partial
D_{\varphi}\tilde{\mathcal{B}}f = \sum_{n=1}^\infty z^{n-1}f_n \sqrt{\varphi_n} \frac{\varphi_{n-1}}{\varphi_{n}} = \sum_{n=1}^\infty z^{n-1} f_n \sqrt{\varphi_{n-1}}\sqrt{\frac{\varphi_{n-1}}{\varphi_{n}}}. \]

Thus, we have
\[ a^* \left( \sum_{n=0}^\infty h_{n} f_{n}  \right) = \sum_{n=0}^\infty \sqrt{\frac{\varphi_{n}}{\varphi_{n+1}}} h_{n+1} f_{n},\quad a \left( \sum_{n=0}^\infty h_{n} f_{n}  \right) = \sum_{n=1}^\infty \sqrt{\frac{\varphi_{n-1}}{\varphi_{n}}} h_{n-1} f_{n}.\]

This give us the action of the raising and lowering operators $a^*$ and $a$ on the Hermite functions as
\begin{align}\label{eq:14}
	a^*h_{n-1} & = \sqrt{\frac{\varphi_{n-1}}{\varphi_{n}}}h_n , \qquad
	a h_{n}  = \sqrt{\frac{\varphi_{n-1}}{\varphi_{n}}}h_{n-1}, \qquad n=1,2,\ldots %\label{eq:15}
\end{align}

Let us emphasize that $\B$ acts now as an intertwining operator in the following way
\begin{align}\label{eq:15}
	\B af &= %\partial
	D_\varphi \B f , \qquad
\B a^*f=z\B f.
\end{align}

Let us give two concrete examples. The first example is not really correct  in our setting since the involved function $\varphi$ is not entire and, hence, the transform is not linked to the Fock space, but to the Hardy space. However, it provides the simple case in which the operators $a$ and $a^*$ appear as forward and backward shift operators, where the corresponding integral transform maps only into the space of analytic functions over the unit disk.
\begin{example}
	Let $\varphi(z) = \frac{1}{1-z}$, where $\varphi_k=1, k=0,1,2,\ldots$. Then from \eqref{eq:14} we have:
	\[ 	a^*h_{n-1} = h_n,\quad 	a h_n = h_{n-1}. \]
\end{example}

Let us now consider a case which fits into our setting.

\begin{example}\label{ex:bt2}
Let us consider $\varphi(z) = \sum_{n=0}^\infty \frac{z^n}{\Gamma'(n+1)}$, where $\varphi_n= 	\frac{1}{\Gamma'(n+1)}, n=0,1,2,\ldots$. 	
	Recall now that $\Gamma(x+1) = x\Gamma(x), x >0$. Using the Digamma function
	$\psi$ we have
	\[ \psi(x) = \frac{d}{d x}\ln[\Gamma(x)]= \frac{ \Gamma'(x) }{\Gamma(x)}, \quad x >0.\]	
	
	The Digamma function is related to the harmonic numbers $H_0=1, H_n = \sum_{k=1}^n \frac{1}{k}, n \in \mathbb{N},$ by
	\[ \psi(n) =  -\gamma + H_{n-1} \quad \Rightarrow \quad   \Gamma'(n+1)  = \Gamma(n+1) [-\gamma + H_{n}] = n! (-\gamma + H_{n}), \]
	where $\gamma$ is the  Euler–Mascheroni constant. Hence,
	\[\frac{\varphi_{n-1}}{\varphi_{n}} = \frac{n!}{(n-1)!}\frac{\left[- \gamma + H_{n}\right]}{\left[- \gamma + H_{n-1}\right]} := n c_n, \quad n=1, 2, \ldots \]
	where the constants $c_n := \frac{\left[- \gamma + H_{n}\right]}{\left[- \gamma + H_{n-1}\right]}$ are such that $\lim_{n\to\infty}c_n=1$. So we get:
	\[ 	a^* h_{n-1} = \sqrt{n-1} c_n h_n,\quad 	a h_n = \sqrt{n-1} c_n h_{n-1}. \]
\end{example}

%%%%%%%%%%%%%%%%
%%%%%%%%%%%%%%%%

\section{Density Theorems for Sampling and Interpolations} Using the general framework of the Fock spaces with respect to the Gelfond-Leontiev operator of generalized differentiation, we extend density results of K. Seip \cite{K92} to our setting. This will later on allow us to obtain lattice conditions for frames arising from the corresponding integral transform over $L^2(\mathbb{R})$. Unfortunately, the non-existence of ``good'' quasi-periodic functions - other than the Weierstrass-$\sigma$ function - leads to the need to adapt K. Seip's methods using controlled approximations.\\

In order to present our theorems, we recall the following definitions for sampling and interpolation sets that include the notion of a weight function $K_\varphi$ in order to match our setting.

A discrete set $\Gamma=\{z_j | z_j\in \mathbb C, ~j \in  \mathbb{J} \}$ is a \text{sampling set} of $\cF_{\varphi}$ if it satisfies an appropriated frame condition, that is, if there exists $0< A \leq B < \infty$  such that
\begin{equation}
	\label{eq:sampling_set}
A\|f \|^2_{\cF, \varphi} \leq \sum_{j \in \Lambda} K_\varphi(- |z_j|^2)|f(z_j)|^2\leq B \|f \|^2_{\cF, \varphi}, \quad \mbox{\rm for all }f\in \cF_\varphi.
\end{equation}

The set $\Gamma$ is an \text{interpolation set} of $\cF_\varphi$ if for every $\ell^2$-sequence $(a_j )_{j\in\Lambda}$  satisfying to the growth condition $\sum_{j \in \Lambda} |a_j|^2 K_\varphi(- |z_j|^2)| < \infty,$  there exists $f\in\cF_\varphi$ such that $f(z_j) = a_j, ~j \in \mathbb{J}$.

For a uniformly discrete set $\Gamma$, using Landau's generalizations of Beurling densities we define the upper and lower uniform densities respectively by:
\[ D^+(\Gamma) = \lim_{r\to \infty} \sup \frac{n^+(r)}{2\pi r^2}  \quad \rm{ and }\quad  D^-(\Gamma) = \lim_{r\to \infty} \inf \frac{n^-(r)}{2\pi r^2} \]
where $n^{\pm}$ represent the smallest and largest number of points of $\Gamma$ in a translate $rI$, where $I$ is a fixed compact set of measure $1$.
Then for uniform discrete sets $\Gamma, \Gamma'$ we can extend the following theorems to our framework of generalized Fock spaces by considering two conditions:
\begin{enumerate}[(i)]
	\item For $g(z)$ defined in \ref{eq:g},$|K_\varphi(-|z|^2)g(z)|$ is quasi periodic.
	\item There exists subharmonic $\phi$ and $K_\varphi(-|z|^2)$ is monotone, such that,
	\begin{itemize}
		\item $\phi(z)=-\ln(K_\varphi(-|z|^2))$
		\item 	$\partial_{\bar{z}} \left(\frac{\partial_z K_{\varphi}(-|z|^2)}{K_\varphi(-|z|^2)}\right) 
				\sim \frac{1}{|z|^2}.$
		
	\end{itemize}

\end{enumerate}

Under either condition (i) or (ii), we have the following results.

First of all, we have a characterization of an interpolation set:

\begin{theorem}
	\label{th:sam_and_inter_2}
	Under condition (i) or (ii),
	there exists $\beta_\varphi$ such that 
	$\Gamma$ is a set of interpolation for $\cF_{\varphi}$ if and only if $D^{+}(\Gamma)<\beta_\varphi$.
\end{theorem}

For the characterization on a sampling ser we have the following theorems.
\begin{theorem}
\label{th:sam_and_inter_1}
Under condition (i) or (ii),
there exists $\beta_\varphi$ such that
$\Gamma$ is a set of sampling for $\cF_{\varphi}$ if and only if it can be expressed as a finite union of uniformly discrete sets and contains a subset $\Gamma'$ s.t. $D^{-}(\Gamma')> \beta_\varphi$.
\end{theorem}

\begin{theorem}
\label{th:sam_and_inter_3}
Under condition (i) or (ii),
$\Gamma$ is a sampling set for $\cF_\varphi$ if and only if can
it contains a
%uniformly discrete
subset $\Gamma'$ and $D^{-}(\Gamma')>{\beta_\varphi}$.
\end{theorem}

Proofs of these theorems are given in the following section.

\subsection{Necessary Conditions for Interpolation and Sampling}
We provide a generalization of the necessary condition for interpolation and sampling results, ,namely Lemma 4.1 and Theorem 5.1 receptively, from \cite{Escudero}.

\begin{lemma}\label{lemma:1.1}
	Let multiplicities pair $(\Lambda,m_{\lambda})$, $m_\lambda\to \mathbb{N}$, $\lambda\in \Lambda\subseteq \mathbb{C}$ and $\{c_{\ell} \;:\; \ell =0,\dots , m_{\Lambda}(\lambda)-1  \}\subset \mathbb{C}$.
	Then for  $\epsilon>0$, there exists $f_\lambda: B(\lambda,\epsilon)\to \mathbb{C}$ analytic for all $z\in B(\lambda,\epsilon)$, the disk centered at $\lambda$ and radius $\epsilon$, such that
	\[ D_\varphi^{(j)}  (f_\lambda (\lambda)) = c_j, \quad 0\leq j\leq m_\Lambda(\lambda) -1,\]
	and 
	\[ |f_\lambda (z)|^2 \leq C_\epsilon \sum_{k=0}^{m_{\Lambda}(\lambda)-1} |c_k|^2 . \]
\end{lemma}
\begin{proof}
	Consider the function
	\[ p_N (z) = \sum_{n=0}^{m_{\Lambda}(\lambda)-1} a_n (z-\lambda)^n.  \]
	with $N=m_{\Lambda}(\lambda)-1$. We want to determine $a_n\in\mathbb{R}$ such that 
	\[ D_\varphi^{(n)}  (p_N (\lambda)) = c_n.  \]	
	
	Unfortunately, we cannot simply restrict ourselves to the case of $\lambda=0$ since our fractional derivatives are not translation invariant.
	For $z=\lambda$ we have
	\begin{equation}\label{eq:D_phi}
		\begin{aligned}
			D_\varphi^{(m)} (p_N(\lambda))  &= \sum_{k=m}^{N}\left[  \sum_{n=k}^{N} a_n{n\choose k} (-\lambda)^{n-k}  \right] \frac{\varphi_{k-m}}{\varphi_{k}}\lambda^{k-m}\\
			&= \sum_{k=m}^{N}  \sum_{n=k}^{N} a_n{n\choose k} (-1)^{n-k} (\lambda)^{n-m} \frac{\varphi_{k-m}}{\varphi_{k}}.
		\end{aligned}
	\end{equation}
	Considering the equation for each $m=0,1,2,\dots N\in\mathbb{N}$ we obtain a system of equations that can be written in the following matrix form
	\begin{equation}\label{eq:matrix}
		\begin{pmatrix}
			&p_N(\lambda)\\
			&({D_\varphi}) (p_N(\lambda))\\
			&({D_\varphi})^{(2)} (p_N(\lambda))\\
			&\vdots\\
			&({D_\varphi})^{(N)} (p_N(\lambda)
		\end{pmatrix}
		=
		\underbrace
		{\begin{pmatrix}
				1 & C_{1,1} & C_{1,2} & \dots  & \dots & \dots & C_{1,N}\\
				0 & \frac{\varphi_0}{\varphi_{1}} & C_{2,2} & \dots & \dots  & \dots & C_{2,N}\\ 
				0 & 0 &  \frac{\varphi_0}{\varphi_{2}} & C_{3,3} & \ddots  & \ddots & C_{2,N}\\
				\vdots & \vdots & \vdots & \ddots & \ddots & \ddots & \vdots \\
				0 & 0  & 0 & \dots & \dots & 0 & \frac{\varphi_{0}}{\varphi_{N}}
		\end{pmatrix}}_{:=A}
		\begin{pmatrix}
			&a_0\\
			&a_1\\
			&a_2\\
			&\vdots\\
			&a_N
		\end{pmatrix}
	\end{equation} 
	where 
	\[C_{m,n} = \sum_{k=m}^{n} {n\choose k} (-1)^{n-k} \lambda^{n-m} \frac{\varphi_{k-m}}{\varphi_k}. \] 
	We note that our matrix $A$ is an upper triangular matrix, hence the inverse of $A$ can be found using 
	\[ A^{-1} = \left[ I_{N\times N} + D^{-1} (A-D) \right] D^{-1} \]
	where $D^{-1}$ is the inverse of the diagonal. We get the following inverse
	\begin{equation}\label{eq:Ainv}
		A^{-1} = \begin{pmatrix}
			1 & \frac{\varphi_{1}}{\varphi_{0}}C_{1,1} & \frac{\varphi_{2}}{\varphi_{0}}C_{1,2} & \frac{\varphi_{3}}{\varphi_{0}}C_{1,3} & \dots & \frac{\varphi_{N}}{\varphi_{0}}C_{1,N}\\
			0 & \frac{\varphi_1}{\varphi_0} & \frac{\varphi_{2}}{\varphi_{0}}\frac{\varphi_{1}}{\varphi_{0}}C_{2,2} & \frac{\varphi_{3}}{\varphi_{0}}\frac{\varphi_{1}}{\varphi_{0}}C_{2,3}  &\dots & \frac{\varphi_{N}}{\varphi_{0}}\frac{\varphi_{1}}{\varphi_{0}}C_{2,N}  \\
			0 & 0 & \frac{\varphi_2}{\varphi_0} & \frac{\varphi_{3}}{\varphi_{0}}\frac{\varphi_{2}}{\varphi_{0}}C_{3,3} & \dots &  \frac{\varphi_{N}}{\varphi_{0}}\frac{\varphi_{2}}{\varphi_{0}}C_{3,N}  \\
			\vdots & \vdots & \vdots & \ddots & \ddots & \vdots \\
			0 & 0 & 0 &  \dots & \frac{\varphi_{N-1}}{\varphi_0} & \frac{\varphi_{N}}{\varphi_0}\frac{\varphi_{N-1}}{\varphi_{0}}C_{N-1,N} \\
			0 & 0 & 0 &  \dots & 0 & \frac{\varphi_N}{\varphi_0}
		\end{pmatrix}. 
	\end{equation}
	
	This allows us to find our coefficients $a_n$ for each $0\leq n\leq N$ by simply applying the inverse matrix \eqref{eq:Ainv} to the left of both sides in \eqref{eq:matrix}, which yields the formula
	\begin{equation}\label{eq:an}
		a_n = \frac{\varphi_{n}}{\varphi_{0}} D^{(n)}_{\varphi} (p_N(\lambda)) + \sum_{k=n}^{N} \frac{\varphi_{k+1}}{\varphi_0}\frac{\varphi_{k}}{\varphi_0}C_{n,k}({D_\varphi})^{(k)} (p_N(\lambda)).
	\end{equation}
	and the interpolation polynomial
	$$
	p_N (z)= \sum_{n=0}^{m_{\Lambda}(\lambda)-1} \left[ \frac{\varphi_{n}}{\varphi_{0}} D^{(n)}_{\varphi} (p_N(\lambda)) + \sum_{k=n}^{N} \frac{\varphi_{k+1}}{\varphi_0}\frac{\varphi_{k}}{\varphi_0}C_{n,k}({D_\varphi})^{(k)} (p_N(\lambda))\right] (z-\lambda)^n
	$$
	
	In this way we obtain the following interpolation function $f_\lambda$ given by
		\begin{equation}\label{eq:inter_f}
		f_\lambda (z)= \sum_{n=0}^{m_{\Lambda}(\lambda)-1} \left[ \frac{\varphi_{n}}{\varphi_{0}} c_{n}+ \sum_{k=n}^{m_\Lambda(\lambda) -1} \frac{\varphi_{k+1}}{\varphi_0}\frac{\varphi_{k}}{\varphi_0} 
		\lambda^{k-n}  
		\sum_{l=n}^{k} {k \choose l} (-1)^{k-l} \frac{\varphi_{l-n}}{\varphi_l} c_k 
		\right] 
		(z-\lambda)^n
	\end{equation}

	Additionally, we remark that 
	\[ |a_n|^2 \leq C_\epsilon 	\sum_{k=0}^{n} |c_k|^2 \]
	where $C_\epsilon$ is the norm of the matrix $A^{-1}$ in \eqref{eq:Ainv}, hence we get 
	\[ |f_\lambda (z)|^2 = \sum_{n=0}^{m_{\Lambda}(\lambda)-1} a_n (z-\lambda)^n \leq C_\epsilon \sum_{k=0}^{m_{\Lambda}(\lambda)-1} |c_k|^2 . \]
	%\[\textbf{to check,.}\]
\end{proof}
%}

Following the above lemma, we provide a generalization of Proposition 5.2 in~\cite{Escudero}.
\begin{theorem}\label{prop5.2}
Let $(\Delta,m_{\Delta})$ be a separated set with multiplices that compatible with $\varphi$.
Let $(\Delta,m_{\Delta})$ be an interpolation set of $\cF_{\varphi}(\mathbb{C})$, and $\sup_{\lambda\in\Lambda}m_{\Lambda}(\lambda)=n_\Lambda+1\geq 2$.
Then there exists a seperation and interpolation set $(\tilde{\Delta},m_{\tilde{\Delta}})$ such that $\sup_{\lambda\in\tilde{\Lambda}}m_{\tilde{\Lambda}}(\lambda)=n_\Lambda$ and 
$$ D^{\pm}(\Lambda,m_{\Lambda})=D^{\pm}(\tilde{\Lambda},m_{\tilde{\Lambda}}). $$
\end{theorem}
\begin{proof}
The proof follows similar steps to the original proof whereby we first define the following 	lattice for $\epsilon\in[0,\min \{\rho(\Lambda)/2,1/4 \}]$.
\[ \Lambda_{\max{}} = \left\{ \lambda\in\Lambda \;:\; m_{\Lambda}(\lambda) = \sum_{z\in\Lambda}(m_\Lambda(z))=n_{\Lambda}+1 \right\} .\]
For each $\lambda\in\Lambda_{\max}$ take $\lambda'\in\mathbb{C}$ such that $|\lambda-\lambda'|=\epsilon$, and define $\Lambda'=\{ \lambda'\;:\; \lambda\in\Lambda_{\max} \}$.
Define
\[
m_{\tilde{\Lambda}}=
\begin{cases}
	m_\Lambda(z) & z\in\Lambda,\; m_{\Lambda}(z)\leq n_\Lambda\\
	n_\Lambda(z) & z\in\Lambda,\; m_{\Lambda}(z) = n_\Lambda+1\\
	1 & z\in\Lambda'
\end{cases}
\]
Now we show that if $\tilde{a}\in \ell_\varphi^2$, then there exists $ f \in \cF_{\varphi}$ such that
\begin{enumerate}[(a)]
	\item $|| f ||_{\cF_{\varphi}}\leq C_{\Lambda}\epsilon^{-n_\Lambda}||\tilde a ||_{\ell_\varphi^2}$.
	\item $D_{\varphi}^{(j)}f(\lambda) = \tilde{a}_{(\lambda,k)},~~\text{for each } \lambda\in\Lambda,\; j\in [ 0,\min \{ n_\Lambda-1,m_\lambda(\lambda)-1 \} ].$
	\item $||f-\tilde{a}_{(\cdot,0)}||_{\ell_\varphi^2}\leq C_{\Lambda}\epsilon^{n_\Lambda+1}||f||_{\cF_{\varphi}}.$	
\end{enumerate}	
We start by building the necessary tools to prove (a) and (b). 
We recall the  
interpolation function $f_\lambda$ given by \eqref{eq:inter_f} with the coefficients $c_k$ given as the fractional derivatives by Lemma \ref{lemma:1.1}. In particular for $\lambda'$ such that $|\lambda'-\lambda|=\epsilon$ and $w\in \mathbb{C}$, we have
\begin{equation*}
		\begin{aligned}
			w= f_{\lambda} (z)&= \sum_{n=0}^{m_{\Lambda}-1} \left[ \frac{\varphi_{n}}{\varphi_{0}} c_{n}+ \sum_{k=n}^{m_\Lambda -1} \frac{\varphi_{k+1}}{\varphi_0}\frac{\varphi_{k}}{\varphi_0} 
			\lambda^{k-n}  
			\sum_{l=n}^{k} {k \choose l} (-1)^{k-l} \frac{\varphi_{l-n}}{\varphi_l} c_k 
			\right] 
			(z-\lambda)^n,
		\end{aligned}
\end{equation*}
denote $z=\lambda'$ and $n_\Lambda= m_\Lambda -1$, we obtain
\begin{align*}
	w
&= \sum_{n=0}^{n_{\Lambda}} %\left[
 \frac{\varphi_{n}}{\varphi_{0}}D^{(n)}_\varphi(f_\lambda(\lambda))
% c_{m_\Lambda -1}
%\right] 
(\lambda'-\lambda)^n
+ \sum_{n=0}^{n_{\Lambda}} \left[ \sum_{k=n}^{n_\Lambda } \frac{\varphi_{k+1}}{\varphi_0}\frac{\varphi_{k}}{\varphi_0} 
\lambda^{k-n}  
\sum_{l=n}^{k} {k \choose l} (-1)^{k-l} \frac{\varphi_{l-n}}{\varphi_l} D^{(k)}_\varphi(f_\lambda(\lambda))%c_k 
%(\lambda')^{n-k}  \sum_{l=k}^{n} { n\choose l} (-1)^{n-l} \frac{\varphi_{l-k}}{\varphi_l} c_k 
\right] 
(\lambda'-\lambda)^n,
%flipe z-lambda to lmabda-lambda'
\end{align*}
which can be written equivalently as 
\begin{align*}
	w &=  
	\frac{\varphi_{n_\Lambda}}{\varphi_0}D^{(n_\Lambda)}_\varphi f_\lambda(\lambda) 	(\lambda'-\lambda)^{(n_\Lambda)} + 
	\sum_{n=0}^{n_{\Lambda}-1}	\frac{\varphi_{n}}{\varphi_{0}}D^{(n)}_\varphi(f_\lambda(\lambda))	(\lambda'-\lambda)^n\\
	&\quad + \sum_{n=0}^{n_{\Lambda}} \left[ \sum_{k=n}^{n_\Lambda } \frac{\varphi_{k+1}}{\varphi_0}\frac{\varphi_{k}}{\varphi_0} 
	\lambda^{k-n}  
	\sum_{l=n}^{k} {k \choose l} (-1)^{k-l} \frac{\varphi_{l-n}}{\varphi_l} D^{(k)}_\varphi(f_\lambda(\lambda))%c_k 
	%(\lambda')^{n-k}  \sum_{l=k}^{n} { n\choose l} (-1)^{n-l} \frac{\varphi_{l-k}}{\varphi_l} c_k 
	\right] 
	(\lambda'-\lambda)^n,
	\\
	&=\frac{\varphi_{n_\Lambda}+\varphi_{n_\Lambda+1}}{\varphi_0}D^{(n_\Lambda)}_\varphi f_\lambda(\lambda) 	(\lambda'-\lambda)^{(n_\Lambda)}\\
	&\quad  + 
	\sum_{n=0}^{n_{\Lambda}-1} \left[\frac{\varphi_{n}}{\varphi_{0}}D^{(n)}_\varphi(f_\lambda(\lambda))+ \sum_{k=n}^{n_\Lambda } \frac{\varphi_{k+1}}{\varphi_0}\frac{\varphi_{k}}{\varphi_0} 
	\lambda^{k-n}  
	\sum_{l=n}^{k} {k \choose l} (-1)^{k-l} \frac{\varphi_{l-n}}{\varphi_l} D^{(k)}_\varphi(f_\lambda(\lambda))
	\right] 
	(\lambda'-\lambda)^n\\ 
	&=\frac{\varphi_{n_\Lambda}+\varphi_{n_\Lambda+1}}{\varphi_0}D^{(n_\Lambda)}_\varphi f_\lambda(\lambda) 	(\lambda'-\lambda)^{(n_\Lambda)}+  D^{(n_\Lambda)}_\varphi(f_\lambda(\lambda))\left[ \sum_{n=0}^{n_{\Lambda}-1} \frac{\varphi_{n_\Lambda+1}}{\varphi_0}\frac{\varphi_{n_\Lambda}}{\varphi_0} 
	\lambda^{n_\Lambda-n}  
	\sum_{l=n}^{n_\Lambda} {n_\Lambda \choose l} (-1)^{n_\Lambda-l} \frac{\varphi_{l-n}}{\varphi_l} 
	(\lambda'-\lambda)^n\right] \\
	&\quad  + 
	\sum_{n=0}^{n_{\Lambda}-1} \left[\frac{\varphi_{n}}{\varphi_{0}}D^{(n)}_\varphi(f_\lambda(\lambda))+ \sum_{k=n}^{n_\Lambda-1 } \frac{\varphi_{k+1}}{\varphi_0}\frac{\varphi_{k}}{\varphi_0} 
	\lambda^{k-n}  
	\sum_{l=n}^{k} {k \choose l} (-1)^{k-l} \frac{\varphi_{l-n}}{\varphi_l} D^{(k)}_\varphi(f_\lambda(\lambda))
	\right] 
	(\lambda'-\lambda)^n 
\end{align*}
Hence we can isolate the highest power $D^{(n_\Lambda)}_\varphi(f_\lambda(\lambda))$

\begin{align*}
	D^{(n_\Lambda)}_\varphi(f_\lambda(\lambda)) = 
	\frac{w-\sum\limits_{n=0}^{n_{\Lambda}-1} \left[\frac{\varphi_{n}}{\varphi_{0}}D^{(n)}_\varphi(f_\lambda(\lambda))+ \sum\limits_{k=n}^{n_\Lambda-1 } \frac{\varphi_{k+1}}{\varphi_0}\frac{\varphi_{k}}{\varphi_0} 
		\lambda^{k-n}  
		\sum\limits_{l=n}^{k} {k \choose l} (-1)^{k-l} \frac{\varphi_{l-n}}{\varphi_l} D^{(k)}_\varphi(f_\lambda(\lambda))
		\right] 
		(\lambda'-\lambda)^n}{\frac{\varphi_{n_\Lambda}+\varphi_{n_\Lambda+1}}{\varphi_0}(\lambda'-\lambda)^{n_\Lambda} + \left[ \sum\limits_{n=0}^{n_{\Lambda}-1} \frac{\varphi_{n_\Lambda+1}}{\varphi_0}\frac{\varphi_{n_\Lambda}}{\varphi_0} 
		\lambda^{n_\Lambda-n}  
		\sum\limits_{l=n}^{n_\Lambda} {n_\Lambda \choose l} (-1)^{n_\Lambda-l} \frac{\varphi_{l-n}}{\varphi_l} 
		(\lambda'-\lambda)^n\right] }.
\end{align*}

For each $\lambda\in\Lambda$ we can rewrite the above line by defining $b_\lambda$ as the a combination of the highest power terms with $\tilde{a}_{(\lambda,j)}$ such that by Lemma \ref{lemma:1.1}. $D_{\varphi}^{(j)}f(\lambda) = \tilde{a}_{(\lambda,j)}$ (with $w=f(\lambda')=\tilde{a}_{(\lambda',0)}$ and $D_\varphi^{(n_\Lambda)} = b_\lambda = \tilde{a}_{(\lambda',n_\Lambda)} $). We define
\begin{align*}
	b_\lambda = 	\frac{\tilde{a}_{(\lambda',0)} -\sum\limits_{n=0}^{n_{\Lambda}-1} \left[ \frac{\varphi_{n}}{\varphi_{0}}D^{(n)}_\varphi(f_\lambda(\lambda))+ \sum\limits_{k=n}^{n_\Lambda-1 } \frac{\varphi_{k+1}}{\varphi_0}\frac{\varphi_{k}}{\varphi_0} 
		\lambda^{k-n}  
		\sum\limits_{l=n}^{k} {k \choose l} (-1)^{k-l} \frac{\varphi_{l-n}}{\varphi_l} 
		\tilde{a}_{(\lambda,k)} 
		\right] 
		(\lambda'-\lambda)^n}{\frac{\varphi_{n_\Lambda}+\varphi_{n_\Lambda+1}}{\varphi_0}(\lambda'-\lambda)^{n_\Lambda} + \left[ \sum\limits_{n=0}^{n_{\Lambda}-1} \frac{\varphi_{n_\Lambda+1}}{\varphi_0}\frac{\varphi_{n_\Lambda}}{\varphi_0} 
		\lambda^{n_\Lambda-n}  
		\sum\limits_{l=n}^{n_\Lambda} {n_\Lambda \choose l} (-1)^{n_\Lambda-l} \frac{\varphi_{l-n}}{\varphi_l} 
		(\lambda'-\lambda)^n\right] }
\end{align*}

This means that we have
\begin{align*}
	b_\lambda
	= \frac{1}{ \left(\frac{\varphi_{n_\Lambda}+\varphi_{n_\Lambda+1}}{\varphi_0} \right)(\lambda'-\lambda)^{n_\Lambda} }
	&\bigg( \tilde{a}_{(\lambda',0)} -  b_\lambda 
	\left[ \sum_{n=0}^{n_{\Lambda}-1} \frac{\varphi_{n_\Lambda+1}}{\varphi_0}\frac{\varphi_{n_\Lambda}}{\varphi_0} 
	\lambda^{n_\Lambda-n}  
	\sum_{l=n}^{n_\Lambda} {n_\Lambda \choose l} (-1)^{n_\Lambda-l} \frac{\varphi_{l-n}}{\varphi_l} 
	(\lambda'-\lambda)^n\right] 
	\\
	- \sum_{n=0}^{n_{\Lambda}-1} &\left[ \frac{\varphi_{n}}{\varphi_{0}}\tilde{a}_{(\lambda,n )} +\sum_{k=n}^{n_\Lambda-1 } \frac{\varphi_{k+1}}{\varphi_0}\frac{\varphi_{k}}{\varphi_0} 
	\lambda^{k-n}  
	\sum_{l=n}^{k} {k \choose l} (-1)^{k-l} \frac{\varphi_{l-n}}{\varphi_l} 
	\tilde{a}_{(\lambda,k)}
	\right] 
	(\lambda'-\lambda)^n\bigg)
\end{align*}

Hence we can obtain the inequality 
\begin{align*}
	|b_\lambda|^2
&\leq
	\underbrace{	2\left| \frac{1}{ \left(\frac{\varphi_{n_\Lambda}+\varphi_{n_\Lambda+1}}{\varphi_0} \right)(\lambda'-\lambda)^{n_\Lambda} }\right|^2}_{\leq Ce^{-2n_\Lambda}}
	\bigg(
	2 |\tilde{a}_{(\lambda',0)}|^2 
	+  |b_\lambda |^2
	\left| \sum_{n=0}^{n_{\Lambda}-1} \frac{\varphi_{n_\Lambda+1}}{\varphi_0}\frac{\varphi_{n_\Lambda}}{\varphi_0} 
	\lambda^{n_\Lambda-n}  
	\sum_{l=n}^{n_\Lambda} {n_\Lambda \choose l} (-1)^{n_\Lambda-l} \frac{\varphi_{l-n}}{\varphi_l} 
	(\lambda'-\lambda)^n\right|^2
	\\
	&\qquad	+ \left| \sum_{n=0}^{n_{\Lambda}-1} \left[ \frac{\varphi_{n}}{\varphi_{0}} \tilde{a}_{(\lambda,n )} +\sum_{k=n}^{n_\Lambda-1 } \frac{\varphi_{k+1}}{\varphi_0}\frac{\varphi_{k}}{\varphi_0} 
	\lambda^{k-n}  
	\sum_{l=n}^{k} {k \choose l} (-1)^{k-l} \frac{\varphi_{l-n}}{\varphi_l} 
	\tilde{a}_{(\lambda,k)}
	\right] 
	(\lambda'-\lambda)^n\right|^2
	\bigg).
\end{align*}
Furthermore, we can estimate 
\begin{align*}
		&\left| \sum_{n=0}^{n_{\Lambda}-1} \frac{\varphi_{n_\Lambda+1}}{\varphi_0}\frac{\varphi_{n_\Lambda}}{\varphi_0} 
	\lambda^{n_\Lambda-n}  
	\sum_{l=n}^{n_\Lambda} {n_\Lambda \choose l} (-1)^{n_\Lambda-l} \frac{\varphi_{l-n}}{\varphi_l} 
	(\lambda'-\lambda)^n\right|^2\\
	&\leq 
	 \sum_{n=0}^{n_{\Lambda}-1} 2^{n+1} \left| \frac{\varphi_{n_\Lambda+1}}{\varphi_0} \right|^2\left|\frac{\varphi_{n_\Lambda}}{\varphi_0} \right|^2
	\left|\lambda^{n_\Lambda-n}  \right|^2
	\left|
	\sum_{l=n}^{n_\Lambda} {n_\Lambda \choose l} (-1)^{n_\Lambda-l} \frac{\varphi_{l-n}}{\varphi_l} 
\right|^2 \epsilon^{-2n} = C'(\epsilon).
\end{align*}
So we have
\[ 
|b_\lambda|^2\leq \frac{2C\epsilon^{-2n_\Lambda}}{1-C'(\epsilon)}\left( 2 |\tilde{a}_{(\lambda',0)}|^2 + 
 \left| \sum_{n=0}^{n_{\Lambda}-1} \left(\frac{\varphi_{n}}{\varphi_{0}}\tilde{a}_{(\lambda,n )}+ \sum_{k=n}^{n_\Lambda-1 } \frac{\varphi_{k+1}}{\varphi_0}\frac{\varphi_{k}}{\varphi_0} 
\lambda^{k-n}  
\sum_{l=n}^{k} {k \choose l} (-1)^{k-l} \frac{\varphi_{l-n}}{\varphi_l} 
\tilde{a}_{(\lambda,k)}
\right) 
\right|^2 \epsilon^{2n}
  \right),
\]
which means
\[
|b_\lambda|^2\leq
C\epsilon^{-2n}
\left(  |\tilde{a}_{(\lambda',0)}|^2+ \sum_{j=0}^{n_\Lambda-1}|\tilde{a}_{(\lambda,j)}|^2 \right) .
\]

Let such such sequence of $\tilde{a}\in\ell_\varphi^2$, we want to find $a\in\cF_{\varphi}$ such that
\[ ||a||^2_{\ell_\varphi^2} \leq C \epsilon^{-2n_\Lambda}||\tilde{a}||^2_{\ell_\varphi^2} <\infty.\]
Using this $b_\lambda$, we define the new sequence $a=\{a_{\lambda,j}\}_{\lambda\in\Lambda}$, by 
\[ a_{(\lambda,j)} = 
\begin{cases}
	b_\lambda& \lambda\in\Lambda\quad\&\quad j=n_{\lambda}\\
	\tilde{a}_{(\lambda,j)}& \text{else}.
\end{cases} 
\]
So we have 
\begin{equation}\label{eq:56}
	||a||^{2}_{\ell_\varphi^2} \leq C \epsilon^{-2n_{\Lambda}} ||\tilde{a}||^{2}_{\ell_\varphi^2}<\infty.
\end{equation} 
Since $(\Lambda,m_{\Lambda})$ is an interpolation set, we have the following properties
\begin{enumerate}[(i)]
	\item $||f||_{\cF_{\varphi}}\leq C_{\Lambda}||a||_{\ell_\varphi^2} .$
	\item $D_{\varphi}^{(j)}f(\lambda) = \tilde{a}_{(\lambda,k)},~~\text{for each } \lambda\in\Lambda,\; j\in [ 0,\min \{ n_\Lambda-1,m_\lambda(\lambda)-1\} ].$
	\item $D_{\varphi}^{(n_{\Lambda})}f(\lambda) =  b_\lambda$ if $\lambda\in\Lambda_{\max}$.
\end{enumerate}
From \eqref{eq:56} and (i),(ii), we get (a),(b).
Part (c) follows from regular Taylor series argument.

Namely
\[ |f(\lambda') - \tilde{a}_{(\lambda',0)} |^2 \leq  C\epsilon^{2(n_\Lambda+1)} ||f||^2_{L^2_\varphi(B(\lambda',1))}\]
which implies 
\[ 
||f-\tilde{a}_{(\lambda',0)}||_{\ell^2_\varphi} \leq C e^{n_\Lambda+1}||f||_{\cF_{\varphi}}.
\]

Showing that $(\tilde{\Lambda},m_{\tilde{\Lambda}})$ is an interpolation set follows the same steps as in \cite{Escudero}.
Finally, $D^{\pm}(\Lambda,m_{\Lambda})=D^{\pm}(\tilde{\Lambda},m_{\tilde{\Lambda}})$ is again by the same arguments as in \cite{Escudero}.

\end{proof}

\begin{theorem}%prop5.3
	\label{th:4.5}
	Let $\lambda,\lambda'\in \mathbb{C}$ and $\epsilon\in(0,0.25)$. If $f\in\cF_{\varphi}$ and $|\lambda-\lambda'|=\epsilon$, there exists a constant $C_\epsilon$ such that 
	\[ |D_\varphi^{(n_\Lambda)}f(\lambda)|^2\leq C_\epsilon\left[ |f(\lambda')|^2 +\sum_{j=0}^{n_\Lambda-1}|D_\varphi^{(j)}f(\lambda)| \right]  + \epsilon||f||^2_{L^2_\varphi(B(\lambda,1)) }.   \]
\end{theorem}

\begin{proof}
	Using the fractional Taylor series around $\lambda$, i.e.
		\begin{equation}\label{eq:inter_f4}
		f (z)= \sum_{n=0}^{\infty} \left[ \frac{\varphi_{n}}{\varphi_{0}} D^{(n)}_\varphi(f(\lambda)) + \sum_{k=n}^{\infty} \frac{\varphi_{k+1}}{\varphi_0}\frac{\varphi_{k}}{\varphi_0} 
		\lambda^{k-n}  
		\sum_{l=n}^{k} {k \choose l} (-1)^{k-l} \frac{\varphi_{l-n}}{\varphi_l} D^{(k)}_\varphi(f(\lambda))
		\right] 
		(z-\lambda)^n.
	\end{equation}

and using similar estimates as in the proof of the previous theorem, the result follows immediately.	
\end{proof}

\begin{theorem}
Let $(\Gamma,m_\Gamma)$ be a separated multiplicative set, such that it is a sampling set  for $\cF_{\varphi}$ and $\sup_{\lambda\in\Lambda} m_\Lambda(\lambda)=n_\lambda+1\geq 2$.
Then there exists another separated multiplicative sampling set $(\tilde{\Lambda},m_{\tilde{\Lambda}})$ such that $\sup_{\lambda\in\tilde{\Lambda}}m_{\tilde{\Lambda}}(\lambda) = n_\Lambda$ and 
\[ D^{\pm}(\Lambda,m_\Lambda) = D^{\pm}(\tilde{\Lambda},m_{\tilde{\Lambda}}).  \]
\end{theorem}

\begin{proof}
	Taking $(\Lambda, m_\Lambda)$ to be a sampling set, it follows 
	\[ ||f||_{\cF_{\varphi}} \leq C_\Lambda \sum_{\lambda\in \Lambda} \sum_{j=0}^{m_\Lambda-1} |D_\varphi^{(j)} f(\lambda)|^2 .\]
	We can get a lower bound for $\lambda \in \Lambda_{\rm{max}}:$ $m_{\tilde\lambda}(\lambda) = n_\lambda (\lambda)-1$. It follows from previous parts
	\[ \sum_{j=0}^{m_\Lambda(\lambda)-1 } |D^{(j)}_\varphi f(\lambda)|^2\leq C_{\varphi,n_\Lambda,\epsilon} \left( |f(\lambda')|+\sum_{j=0}^{m_\Lambda-1} |D_{\varphi}^{(j)} f(\lambda)|^2  \right) +\epsilon||f||_{\ell^2_\varphi(B(0,1))} .\]
	Then the rest of the proof for the lower bound, and upper bound would follow same steps as in \cite{Escudero}.
\end{proof}

\section{ A Fractional Analogue of the Weierstrass-$\sigma$ function}\label{sec:weie}

We propose the following analogue for the Weierstrass-$\sigma$ function endowed with the Gelfond-Leontiev operator of generalized differentiation. %; as the fractional Weiestrass-$\sigma$ function. %(with $p=2$) to our framework.

Let $\Lambda = \{ \lambda_{m,n} : \lambda_{m,n} = \lambda (m+in), m, n \in \BZ, \lambda >0 \}$ be a square lattice. Let us consider a lattice $\Gamma=\{z_{m,n}, m, n \in \BZ \}$ which is uniformly closed to $\Lambda$, i.e. for which there exist constants $Q$ and $q(\Gamma)$ such that
$$|z_{m,n}-\lambda_{m,n}|<Q,$$
and
$$q(\Gamma) = \inf\limits_{(m,n)\neq (k,l)}|z_{m,n}-z_{k,l}|>0.$$

In the classic case the Weierstrass-$\sigma$ function associated to the original square lattice $\Lambda$ is given by
%related to $\Gamma=\{z_{mn}\}$, which uniformly close to a square lattice $\Lambda$ where $\lambda_{mn} = \beta_n(m+in)$ with $z_{00}$ as the $0$ point:
\begin{equation*}
	\label{eq:sig}
	\sigma(z; \Lambda) := z \prod_{(m,n) \not= (0,0)}\left(1-\frac{z}{\lambda_{m,n}} \right) e^{\left( \frac{z}{\lambda_{m,n}}+\frac{1}{2}\frac{z^2}{\lambda^2_{m,n}} \right)} = z \prod_{(m,n) \not= (0,0)}E_{2, \sigma}(z; m, n),
\end{equation*} where
$$E_{2, \sigma}(z; m, n) =   \left(1-\frac{z}{\lambda_{m,n}} \right) e^{\left( \frac{z}{\lambda_{m,n}}+\frac{1}{2}\frac{z^2}{\lambda^2_{m,n}} \right)}$$

In this work, we wish to construct a function which generalizes the Weierstrass-$\sigma$ function.

In order to obtain such a function we assume $\varphi_0=1.$ Given
\begin{equation}
	\label{eq:norm}
	\varphi (z) = 1+\sum_{k=1}^{\infty}\varphi_k z^k,
\end{equation}
we define the auxiliar entire function  
$ \psi(z) = \sum_{n=1}^{\infty}\psi_nz^n,$ 
such that the corresponding fractional Weierstrass-$\sigma$ factor $E(z)$ satisfies
\begin{equation}
	\label{eq:analog}
	E(z)=(1-z)\varphi(z\psi_1+z^2\psi_2).
\end{equation}

%%%%%%%%%%%%%%
Developing we get
\begin{eqnarray}
	E(z) & = & (1-z) \sum_{n=0}^\infty \varphi_n z^n (\psi_1+z\psi_2)^n \nonumber \\
	& = & (1-z) \sum_{n=0}^\infty \sum_{k=0}^n {n \choose k} \varphi_n   \psi_1^{n-k} \psi_2^k  z^{n+k} \nonumber \\
	& = &  \sum_{n=0}^\infty \sum_{k=0}^n {n \choose k} \varphi_n   \psi_1^{n-k} \psi_2^k  z^{n+k} - \sum_{n=0}^\infty \sum_{k=0}^n {n \choose k} \varphi_n   \psi_1^{n-k} \psi_2^k  z^{n+k+1} \nonumber \\
	& = &  \sum_{m=0}^\infty \sum_{k=0}^{\lfloor m/2 \rfloor} {m-k \choose k} \varphi_{m-k}   \psi_1^{m-2k} \psi_2^k  z^{m} - \sum_{m=1}^\infty \sum_{k=0}^{\lfloor (m-1)/2 \rfloor} {m-1-k \choose k} \varphi_{m-1-k}   \psi_1^{m-1-2k} \psi_2^k  z^{m} \nonumber \\
	& = &   \varphi_{0}  + \sum_{m=1}^\infty \left[ \sum_{k=0}^{\lfloor m/2 \rfloor} {m-k \choose k} \varphi_{m-k}   \psi_1^{m-2k} \psi_2^k - \sum_{k=0}^{\lfloor (m-1)/2 \rfloor} {m-1-k \choose k} \varphi_{m-1-k}   \psi_1^{m-1-2k} \psi_2^k  \right] z^{m} .
\end{eqnarray}
Hence, we obtain (recall $\varphi_0=1$)
\begin{eqnarray}
	E(z) & = & 1 + (\varphi_{1}   \psi_1 - 1) z + (\varphi_{2}   \psi_1^2 + \varphi_{1}   \psi_2 - \varphi_1 \psi_1 ) z^2  \nonumber \\
	&  & \hspace{1cm} + \underbrace{(\varphi_{3}   \psi_1^3 + 2\varphi_{2} \psi_1  \psi_2 - \varphi_2 \psi_1^2 - \varphi_1 \psi_2 ) z^3 + \cdots	}_{=: \Omega(z) z^3} \label{eq:analog_explicity}
\end{eqnarray}
We further impose the coefficients of $z$ and $z^2$ to be zero, that is
\begin{equation}\label{Psis}
	\psi_1 = \frac{1}{\varphi_1}, \quad    \psi_2 =  \frac{\varphi_1^2 - \varphi_{2}}{\varphi_1^3},
\end{equation}
so that we get
$| 1- E(z) | = | \Omega(z)| |z|^3,$ where $\Omega(z)$ is the reminder. In the unit disk $|z|<1$ we get the inequality:
\begin{equation}
	\label{eq:E_inequality}
	|1-E(z)| \leq |\Omega(z)|,
\end{equation}
as desired.

We obtain a $\sigma$-function of the form
$$\sigma(z; \Lambda) := z \prod_{(m,n) \not= (0,0)}\left(1-\frac{z}{\lambda_{m,n}} \right) \varphi \left(\psi_1\frac{z}{\lambda_{m,n}}+\psi_2\frac{z^2}{\lambda^2_{m,n}} \right),$$
and to which we associated the auxiliar function $g$ given by
\begin{equation}
	\label{eq:g}
	g(z; \Gamma) := (z-z_{00}) \prod_{(m,n) \not= (0,0)}\left(1-\frac{z}{z_{m,n}} \right) \varphi{\left( \psi_1 \frac{z}{z_{m,n}}+\psi_2\frac{z^2}{\lambda^2_{m,n}} \right)},% = z \prod_{m,n\not= 0}E_2(z; m, n).
\end{equation}
with factors
$$E_\Lambda(z;m,n) =  \left(1-\frac{z}{\lambda_{m,n}} \right) \varphi{\left( \psi_1 \frac{z}{\lambda_{m,n}}+\psi_2 \frac{z^2}{\lambda^2_{m,n}} \right)}, \qquad E_\Gamma(z;m,n) =  \left(1-\frac{z}{z_{m,n}} \right) \varphi{\left(  \psi_1 \frac{z}{z_{m,n}}+ \psi_2  \frac{z^2}{\lambda^2_{m,n}} \right)}.$$

When clear from the context, we will abbreviate to $E(z):= E_{\Gamma / \Lambda}(z; m,n)$.

We can further develop $\Omega = \Omega(z)$ as
\begin{eqnarray*}
	\Omega(z) & = & \sum_{m=3}^\infty \left[ \sum_{k=0}^{\lfloor m/2 \rfloor} {m-k \choose k} \varphi_{m-k}   \psi_1^{m-2k} \psi_2^k - \sum_{k=0}^{\lfloor (m-1)/2 \rfloor} {m-1-k \choose k} \varphi_{m-1-k}   \psi_1^{m-1-2k} \psi_2^k  \right] z^{m-3} \nonumber \\
	& = & \sum_{m=0}^\infty \left[ \sum_{k=0}^{\lfloor (m+3)/2 \rfloor} {m+3-k \choose k} \varphi_{m+3-k}   \psi_1^{m+3-2k} \psi_2^k - \sum_{k=0}^{\lfloor (m+2)/2 \rfloor} {m+2-k \choose k} \varphi_{m+2-k}   \psi_1^{m+2-2k} \psi_2^k  \right] z^{m} \nonumber \\
	& = & \sum_{n=0}^\infty  \sum_{k=0}^{n+1} \left[ \frac{2n+3-k}{2n+3-2k}  \varphi_{2n+3-k}   \psi_1 -   \varphi_{2n+2-k}  \right] {2n+2-k \choose k}  \psi_1^{2n+2-2k} \psi_2^k  z^{2n} \nonumber \\
	&  &+ \sum_{n=0}^\infty \left[ \sum_{k=0}^{n+1} \left( \frac{2n+4-k}{2n+4-2k}  \varphi_{2n+4-k}   \psi_1 -   \varphi_{2n+3-k}  \right) {2n+3-k \choose k}  \psi_1^{2n+3-2k} \psi_2^k  + \varphi_{n+2}   \psi_2^{n+2} \right] z^{2n+1} \nonumber \\
	& = & \sum_{n=0}^\infty  \sum_{k=0}^{n+1} \left[ \frac{2n+3-k}{2n+3-2k}  \frac{\varphi_{2n+3-k}}{ \varphi_1} -   \varphi_{2n+2-k}  \right] {2n+2-k \choose k}  \frac{(\varphi_1^2 -\varphi_2)^k}{\varphi_1^{2n+2+k}}  z^{2n} \nonumber \\
	&  &+ \sum_{n=0}^\infty \left[ \sum_{k=0}^{n+1} \left( \frac{2n+4-k}{2n+4-2k}  \frac{\varphi_{2n+4-k}}{ \varphi_1} -   \varphi_{2n+3-k}  \right) {2n+3-k \choose k}  \frac{(\varphi_1^2 -\varphi_2)^k}{\varphi_1^{2n+3+k}}   + \varphi_{n+2}  \frac{(\varphi_1^2 -\varphi_2)^{n+2} }{\varphi_1^{3n+6}}    \right] z^{2n+1}. \label{Eq:29}
\end{eqnarray*}

For $|z| \leq 1$ we have for $\Omega = \Omega(z)$ an estimate in terms of the coefficients of $\varphi_n, \psi_1,$ and $\psi_2.$
\begin{gather*}
	E(z) = \varphi(z\psi_1+z^2\psi_2) - z \varphi(z\psi_1+z^2\psi_2) = \sum_{n=0}^\infty \varphi_n  z^n(1-z) (\psi_1 -\psi_2 z)^n \\
	\Rightarrow ~ E(z)  -1 = - z + \varphi_1 z(1-z) (\psi_1 -\psi_2 z) + \varphi_2 z^2(1-z) (\psi_1 -\psi_2 z)^2 + \sum_{n=3}^\infty \varphi_n  z^n(1-z) (\psi_1 -\psi_2 z)^n \\
	= (- \varphi_1 \psi_2 + 2 \varphi_2 \psi_1 \psi_2 -\varphi_2 \psi_1^2)z^3 + (\varphi_2 \psi_1 \psi^2_2 -2\varphi_2 \psi_1 \psi_2)z^4  - \varphi_2 \psi_2^2 z^5 + \sum_{n=3}^\infty \varphi_n  z^n(1-z) (\psi_1 -\psi_2 z)^n,
\end{gather*}
and hence, for $|z| \leq 1$ we obtain
\begin{equation}
	\label{eq:Omega}
	\begin{aligned}
		|\Omega(z)|&\leq \left| \varphi_1 \psi_2 - 2 \varphi_2 \psi_1 \psi_2 +\varphi_2 \psi_1^2 |+ |\varphi_2 \psi_1 \psi^2_2 - 2\varphi_2 \psi_1 \psi_2 | +|\varphi_2 \psi_2^2\right|  +2\sum_{n=3}^{\infty}|\varphi_n| \big(|\psi_1|+|\psi_2| \big)^n.
	\end{aligned}
\end{equation}

With the restrictions on $\varphi_1$ and $\varphi_2$ such that $|\psi_1|+|\psi_2|<R,$ where $R$ is the radius of convergence of $\sum\varphi_nz^n,$ 
we obtain the 
lower bound for $R$ given by:
\begin{equation}
	\label{eq:rl}
	R_L=|\psi_1|+|\psi_2|=\frac{2\varphi^2_1-\varphi_2}{\varphi_1^3},
\end{equation}
while the upper bound for $R$ is given by  
\begin{equation}
	\label{eq:ru}
	R_U = \frac{1}{\limsup\limits_{n\to \infty}|\varphi_n|^{\frac{1}{n}}}.
\end{equation}

\begin{example}
	\label{ex:1}
	For $\varphi (z)= e^z = \sum_{n=0}^{\infty}\frac{z^n}{n!},$ that is, $\varphi_n  = \frac{1}{\Gamma(n+1)} = \frac{1}{n!}$, we obtain $\psi_1 = 1$ and $\psi_2 = \frac{1^2-1/2}{1^3}= \frac{1}{2}.$ Hence, for the Weierstrass factor we have
	\[ E(z) = (1-z) e^{z + \frac{z^2}{2}},  \]
	with lower and upper bounds given by  \eqref{eq:rl} and \eqref{eq:ru} as	$R_L= \frac{3}{2}$ and $R_U= \limsup\limits_{n\to\infty}\frac{1}{\left(\frac{1}{n!}\right)^{\frac{1}{n}}}\to\infty$.
\end{example}

\begin{example}
	\label{ex:shift_limits}
	For $\varphi(z)= \frac{1}{1-z} =\sum_{n=0}^{\infty} z^n$ (not an entire function), we have $\psi_1 = 1$, and $\psi_2=0$, which leads to $$E(z) = (1-z) \varphi(z \psi_1 + z^2 \psi_2) = 1$$ and, thus, $\Omega(z) =0.$   Also, our bounds for $R$ are $R_L = (2-1)/1=1$ and $R_U=
	\limsup\limits_{n\to\infty}\frac{1}{1^{\frac{1}{n}}}=1$. This example is also an example of an extreme case where one has $R_U=R_L=R=1$.
\end{example}

\begin{example}
	Consider the Mittag-Leffler function defined in Example \ref{exML}
	$$\varphi (z)=E_{\frac{1}{\rho},\mu}(z) = \sum_{k=0}^{\infty}\frac{z^k}{\Gamma\left(\mu+\frac{k}{\rho}\right)}$$
	for $\rho>0,{\rm Re}(\mu)>0.$
	
	Then, and taking into account the normalization factor in (\ref {eq:norm}),  we have $\psi_1=\frac{\Gamma\left(\mu+\frac{1}{\rho}\right)}{\Gamma\left(\mu \right)}$ and
	$\psi_2= \frac{\Gamma\left(\mu+\frac{1}{\rho}\right)}{\Gamma\left(\mu \right)} -  \frac{ \Gamma^{3}\left(\mu + \frac{1}{\rho} \right)}{\Gamma \left(\mu+\frac{2}{\rho}\right) \Gamma^2(\mu) }.$ Replacing these values in (\ref{eq:analog}) %\eqref{Eq:29}
	we obtain
	$$E(z)=1+z^3\Omega(z) = \Gamma(\mu) (1-z) E_{\frac{1}{\rho},\mu} \left( \frac{\Gamma\left(\mu+\frac{1}{\rho}\right)}{\Gamma\left(\mu \right)} z +  \frac{ \Gamma \left(\mu + \frac{1}{\rho}\right)\Gamma \left(\mu + \frac{2}{\rho}\right) \Gamma \left(\mu \right) -\Gamma^{3}\left(\mu + \frac{1}{\rho} \right)}{\Gamma \left(\mu+\frac{2}{\rho}\right) \Gamma^2(\mu) } z^2  \right).$$ Moreover, the lower bound is then given by	
	\[ R_L = 2\frac{\Gamma\left(\mu+\frac{1}{\rho}\right)}{\Gamma\left(\mu \right)} -  \frac{ \Gamma^{3}\left(\mu + \frac{1}{\rho} \right)}{\Gamma \left(\mu+\frac{2}{\rho}\right) \Gamma^2(\mu) }.\]
	
	For the upper bound, and using Stirling's approximation $\Gamma(x+1)\sim\sqrt{2\pi x}\left(\frac{x}{e}\right)^x$, we have for the normalized coefficients of $\varphi$
	\begin{gather*}
		|\Gamma(\mu) \varphi_n|^{\frac{1}{n}}  = \left| \frac{\Gamma(\mu)}{\Gamma\left(\mu +\frac{n}{\rho} \right)} \right|^{\frac{1}{n}} = \left| \Gamma(\mu)  \right|^{\frac{1}{n}} \left|{\Gamma\left(\mu +\frac{n}{\rho} \right)} \right|^{-\frac{1}{n}}.
	\end{gather*}
	As $|\Gamma(\mu)| >0$ we only analyse the behaviour of the last term.
	\begin{align*}
		\left| \Gamma\left(\mu +\frac{n}{\rho} \right) \right|^{\frac{1}{n}}
		&\sim\left|\sqrt{2\pi\left(\mu+\frac{n}{\rho}-1\right)}\left(\frac{\mu+\frac{n}{\rho}-1}{e}\right)^{\left({\mu+\frac{n}{\rho}-1}\right)} \right|^{\frac{1}{n}}\\
		&=\left|\sqrt{2\pi\left(\mu+\frac{n}{\rho}-1\right)} \right|^{\frac{1}{n}}
		\left|\left(\frac{\mu+\frac{n}{\rho}-1}{e}\right)^{\left(\frac{\mu}{n}-\frac{1}{n}+\frac{1}{\rho}\right)} \right|~~
		\xrightarrow[]{n\to\infty}~~\infty.
	\end{align*}
	Therefore, the upper bound is given by $R_U = \frac{1}{\limsup\limits_{n\to \infty}|\varphi_n|^{\frac{1}{n}}}= \infty$.
\end{example}

\begin{example}\label{ex:eab}
	Let $K_\varphi(x) = e^{ -a(-x)^b}$ for $x>0$, where $b >0$, ${\rm Re}(a)>0.$ Then, we have $\varphi(z)=\sum_{n=0}^\infty \varphi_n z^n,$ where
	$$\varphi_n = \frac{b a^{\frac{n+1}{b}} }{\Gamma\left( \frac{n+1}{b} \right)}, \qquad n =0, 1, 2, \ldots$$
	We have (again taking into account the normalization factor)
	\[
	\tilde \varphi_1=\frac{a^{\frac{1}{b}} \Gamma\left( \frac{1}{b} \right)}{\Gamma\left( \frac{2}{b} \right)},\quad \tilde \varphi_2=\frac{a^{\frac{2}{b}} \Gamma\left( \frac{1}{b} \right)}{\Gamma\left( \frac{3}{b} \right)}.
	\]
	This gives
	$$E(z)= \frac{ \Gamma \left( \frac{1}{b} \right) }{b a^{\frac{1}{b} }}(1-z) { \rm exp} \left[-a \left(- \frac{\Gamma\left( \frac{2}{b} \right)}{a^{\frac{1}{b}} \Gamma\left( \frac{1}{b} \right)} z -    \frac{ 2 \Gamma\left( \frac{1}{b} \right) \Gamma\left( \frac{2}{b} \right) \Gamma\left( \frac{3}{b} \right) - \Gamma^3\left( \frac{2}{b} \right)}{ a^{\frac{1}{b}}  \Gamma^2\left( \frac{1}{b} \right) \Gamma\left( \frac{3}{b} \right)}   z^2 \right)^b \right],$$
	with the following lower bound
	\[ R_L =
	2 \frac{\Gamma\left( \frac{2}{b} \right)}{a^{ \frac{1}{b}} \Gamma\left( \frac{1}{b} \right)}   - \frac{\Gamma^3\left( \frac{2}{b} \right)}{ a^{\frac{1}{b}} \Gamma\left( \frac{3}{b} \right) \Gamma^2\left( \frac{1}{b} \right) }.
	\]
	
	Using Stirling's approximation $\Gamma(x+1)\sim\sqrt{2\pi x}\left(\frac{x}{e}\right)^x$, we have
	\begin{gather*}
		|\varphi_n|^{\frac{1}{n}} = \left| \frac{b a^{\frac{n+1}{b}}}{\Gamma \left( \frac{n+1}{b} \right)} \right|^{\frac{1}{n}} = \left| b a^{\frac{n+1}{b}} \right|^{\frac{1}{n}}  \left|  \Gamma \left( \frac{n+1}{b} \right) \right|^{-\frac{1}{n}}
	\end{gather*} and again
	\begin{align*}
		\left|  \Gamma \left( \frac{n+1}{b} \right) \right|^{\frac{1}{n}} &\sim  \left( \sqrt{2\pi \left(\frac{n+1}{b}-1\right)}\left( \frac{\frac{n+1}{b}-1}{e} \right)^{\frac{n+1}{b}-1} \right)^{\frac{1}{n}}
		\xrightarrow[]{n\to\infty}~~\infty,
	\end{align*}
	so that we obtain $R_U=\infty$.
\end{example}

\begin{example}
	Consider $\varphi(z) = \sum_{k=0}^\infty  	\frac{z^k}{\Gamma^{(n)}(k+1)}$, for some fixed $n \in \mathbb{N}$, associated to the measure $K_\varphi(x)= e^{-x}\ln^n |x|, x>0$. Hence, we obtain for the normalized coefficients
	$$\tilde \varphi_1 = \frac{\Gamma^{(n)}(2)}{\Gamma^{(n)}(1)}, \qquad \tilde \varphi_2 = \frac{\Gamma^{(n)}(3)}{\Gamma^{(n)}(1)},$$ and
	$$E(z) =(1-z) \, \Gamma^{(n)}(1)  \varphi \left( \frac{\Gamma^{(n)}(1)}{\Gamma^{(n)}(2)} z +  \frac{ \Gamma^{(n)}(1)\Gamma^{(n)}(2)  - \Gamma^{(n)}(3) (\Gamma^{(n)}(1) )^2}{(\Gamma^{(n)}(2) )^3 } z^2  \right).$$
	Hence, we obtain the lower bound for $|\Omega(z)|$ as 	
	\[ R_L = 2\frac{\Gamma^{(n)}(1)}{\Gamma^{(n)}(2)} - \frac{\Gamma^{(n)}(3) (\Gamma^{(n)}(1) )^2}{(\Gamma^{(n)}(2) )^3 } \]

	In order to find the upper limit we need to consider the Digamma function $\phi$ like in Example~\ref{ex:bt2}
	\[ R_U = \frac{1}{\lim\sup\limits_{k\to \infty} |\Gamma^{(n)}(k+1)|^{-\frac{1}{k}}} \]
	
	For $n=1$ we have
	\begin{align*}
		|\Gamma'(k+1)|^{\frac{1}{k}}&=| k!(-\gamma+H_k)|^{\frac{1}{k}} = (k! )^{\frac{1}{k}}|(-\gamma+H_k)|^{\frac{1}{k}}\to \infty
	\end{align*}
	Since the dominating term satisfies
	$\displaystyle \lim\limits_{k\to\infty}(k!)^{\frac{1}{k}}\to \infty,$ we get $R_U=\infty$ when $n=1$.
	
	%%%%%%%%%%%%%%%%%%%%%%	
	%%%%%%%%%%%%%%%%%%%%%%	

	For a general $n>1$, using the fact that $\Gamma^{(n)}(x)=\int_{0}^{\infty}t^{x-1}e^{-t}\ln^n(t) dt$ we have
	\begin{align*}
		R_U = \frac{1}{\lim\sup\limits_{k\to \infty} |\Gamma^{(n)}(k)|^{-\frac{1}{k}}}
		&=\frac{1}{\lim\sup\limits_{k\to \infty} \left| \int_{0}^{\infty} t^{k-1} e^{-t}\ln^n(t) dt \right|^{-\frac{1}{k}}}\\
		&\geq \frac{1}{\lim\sup\limits_{k\to \infty} \left(\int_{0}^{\infty} |t^{\frac{k-1}{k}}| |e^{-t}\ln^n(t)|^{\frac{1}{k}} dt \right)^{-1}}
	\end{align*}		
\end{example}
			\begin{align*}
		&= \frac{1}{ \left(\int_{0}^{\infty}\lim\sup\limits_{k\to \infty} |t^{\frac{k-1}{k}}| |e^{-t}\ln^n(t) |^{\frac{1}{k}}dt \right)^{-1} }\\
		&\to \frac{1}{\left(\int_{0}^{\infty} tdt \right)^{-1}}\to \infty.
	\end{align*}
	\textit{Therefore $R_U=\infty$, for all $n\geq1$.}

\begin{remark}
	For every function $\varphi$ we can, with some appropriate restrictions, associate a counterpart of the Weierstrass factors $\prod_{m,n} (1-\frac{z}{z_{mn}})\varphi\left(\psi_1\frac{z}{z_{mn}} +
	\psi_2\frac{z^2}{\lambda^2_{mn}}\right)$.
	This allows us to find classes of entire functions with preassigned zeros associated to $\varphi$.
\end{remark}

\section{Proof of main theorem for the quasi-periodic case}

While the proofs for Theorems \ref{th:sam_and_inter_1}, \ref{th:sam_and_inter_2}, and \ref{th:sam_and_inter_3} follow similar steps as in \cite{K92} and \cite{K292}, they require the new analogue of the Weierstrass-$\sigma$ function endowed with the Gelfond-Leontiev operator of generalized differentiation  introduces in Section \ref{sec:weie}.
We obtain the following estimate for the function $g$.

\begin{lemma}
{For the function $g=g(\cdot;\Gamma)$ defined on the region $R = \{ z, |z|<\frac{1}{2}\beta_\varphi \}$}
such that $|K_\varphi(-|z|^2)g(z)|$ is quasi-periodic,
we have the inequality
\[ \left|K_\varphi \left( -|z|^2 \right) g(z) \right| \geq c {\rm d}(z,\Lambda) \]
with a positive constant $c$.
\end{lemma}

\begin{proof}
From the restriction that  $\left| K_\varphi \left(-  |z|^2 \right) g(z) \right|$ is quasi-periodic, it follows that for any $z$ we have some  $\omega\in \bar{R}$ such that
$$\left| K_\varphi \left(-  |z|^2 \right) g(z) \right|=\left| K_\varphi \left(-  |\omega|^2 \right) g(\omega) \right|.$$
	
{First we note that on the region $R$, $g(z)$ has no zeros of the form $z_{mn}$ for $n\neq m$  other than $z_{00}$. Then,}
using $|1-E(z)|\leq |z|^3 |\Omega (z)|$ from \eqref{eq:E_inequality} and the triangle inequality we have
\begin{align*}
	|1|-|E(z)| \leq |1-E(z)| \leq | z^3\Omega(z) |	 &\implies -|E(z)| \leq | z^3\Omega(z) | - 1\\
	&\implies |E(z)|\geq 1-|z^3\Omega(z)|.
\end{align*}
Then for a constant $c >0$ we have
\begin{align*}
	\left| K_\varphi \left(-  |z|^2 \right) g(z) \right| &= \left| K_\varphi \left( - |z|^2  \right) (z-z_{00}) \prod_{m,n} E_2(z) \right|\\
	&= \left| K_\varphi \left( - |z|^2  \right)\right|\left| (z-z_{00})\right|\prod_{m,n} \left|E_2(z) \right|\\
	&\geq \left| K_\varphi \left( - |z|^2  \right)\right| \left|\prod_{m,n} ( 1- |z^3\Omega(z)| )\right|\left| (z-z_{00})\right|\\
	& = c {\rm d}(z,\Lambda).
\end{align*}
\end{proof}

In addition we can also give a bound for $g(z)$ which
by \eqref{eq:Omega} is independent of $z$.
Therefore, from \eqref{eq:E_inequality} and \eqref{eq:Omega} in \eqref{eq:g} combined with the triangle inequality we obtain the following independent bound of $g(z)$:
\begin{equation}
	\label{eq:g_bound}
	\begin{aligned}
		|g(z)| & = |(z-z_{00})||\prod_{m,n} E_2(z)|,\\
		& \leq |(z-z_{00})|\prod_{m,n}|(1+z^3\Omega(z))| \leq  |z-z_{00}|  |\Omega(z)^N z^{3N} +\beta_1(z)|= |\Omega(z)^N z^{3N+1} +\beta_2(z)|,
	\end{aligned}
\end{equation}
where $N$ is a constant equal to the highest degree when performing the product $\prod\limits_{m,n} (z^3\Omega(z) +1)$. As we only consider the leading term when looking for the upper bound, we denote by $\beta_1(z)$ and $\beta_2(z)$ the
expressions depending of $z$ which contain  all terms with lower degrees than the leading term.

We also need a lower estimate for our generalized $\sigma-$function as given in the next lemma. This lemma is an immediate consequence of the above lemma.
\begin{lemma}\label{lemma:sigma}
For the generalized Weierstrass $\sigma$-function such that $|K_\varphi(-|z|^2)\sigma(z)|$ is quasi-periodic we have the estimate
$$
K_\varphi(-|z|^2)|\sigma(z)|\geq C d(z,\Lambda),
$$
with a constant $C>0$.
\end{lemma}

This allows us to prove the following lemma.

\begin{lemma}\label{lemma:ineq}
	Let $\Gamma$ be uniformly close to the square lattice of density $\beta_\varphi$ and $K_\varphi$ be non-zero and monotonic. Then, there exists constants $c,c_1$, and $c_2$
	such that for all $z$ we have
	\begin{equation}
		\label{eq:inqaulity}
		c_1 \gamma(z) e^{-c|z|\log |z|} \cdot {\rm{dist}}(z,\Gamma) ~\leq~ |K_\varphi(-|z|^2) g(z)|~\leq~  c_2 \gamma(z) e^{c|z|\log|z|}
	\end{equation}
	where
	\begin{equation}\label{eq:gamma}
		\gamma(z) =
		\begin{cases}
			1, & {\rm growth}~(K_\varphi) < {\rm growth}~(e^z)\\
			|K_\varphi(z)|, & \text{otherwise}
		\end{cases},
	\end{equation}
	and for every $z_{mn}\in \Gamma$ we have $\left| K_\varphi(-|z_{m,n}|^2) g'(z_{mn}) \right| \geq c_1 e^{-c|z_{mn}|\log |z_{mn}|}$.
\end{lemma}

\begin{proof}
	In a similar spirit to the proof in \cite{K292} [Page 3, Lemma 2.2] , we write
	$$g(z) = \frac{\sigma_\varphi(z)}{d(z,\Lambda)}d(z,\Gamma)h(z)\implies h(z) = \frac{g(z)}{\sigma_\varphi(z)}\frac{d(z,\Lambda)}{d(z,\Gamma)}.$$
	
	We estimate $h(z)$ by factorizing and estimating each of the components of
	$$h(z) = h_1(z) h_2(z),$$
	where
	\begin{equation}
	\label{eq:h1}
	h_1(z)=\frac{d(z,\Lambda)}{d(z,\Gamma)} \frac{z-z_{0,0}}{z}
	\prod\limits_{m,n} \frac{\left( 1-\frac{z}{z_{m,n}} \right)}{\left( 1 - \frac{z}{\lambda_{m,n}}\right)}.
	\end{equation}
	and
	\begin{equation}
	\label{eq:h2}
	h_2(z)=\prod\limits_{m,n} \frac{\varphi\left( \psi_1\frac{z}{z_{m,n}}+ \psi_2\frac{z^2}{\lambda_{m,n}}  \right)}{\varphi\left(\psi_1 \frac{z}{\lambda_{m,n}}+\psi_2 \frac{z^2}{\lambda_{m,n}}  \right)}
	=\prod\limits_{|z_{m,n}|>2|z|}\left(\dots\right)\cdot\prod\limits_{|z_{m,n}|\leq2|z|}\left(\dots\right).
	\end{equation}	
	
	We estimate $h(z)$ for a non-zero monotonic $\varphi >0$, for $|z_{m,n}|\leq 2|z|$ by finding $a$, $b$ such that
	$$ \varphi\left( \psi_1\frac{z}{\lambda_{m,n}}+ \psi_2\frac{z^2}{\lambda_{m,n}}  \right) \leq a, \qquad \varphi\left( \psi_1 \frac{z}{z_{m,n}}+ \psi_2\frac{z^2}{z_{m,n}}  \right) \geq b.$$

	Since $\varphi$ is non-zero and monotonic, it would be sufficient to find the value of $z$ for which $\varphi(z)$ is minimum and maximum, respectively, by solving the quadratic equation for $z$.
	We get $z_{\rm min} = \frac{\lambda_{m,n}}{\lambda_{m,n}}= -1$, and $z_{\rm max} = -\frac{\lambda_{m,n}}{z_{m,n}}$. So
	$$\prod\limits_{|z_{m,n}|\leq2|z|}\frac{\varphi\left(\psi_1 \frac{z}{z_{m,n}}+\psi_2 \frac{z^2}{\lambda_{m,n}}  \right)}{\varphi\left( \psi_1 \frac{z}{\lambda_{m,n}}+ \psi_2\frac{z^2}{\lambda_{m,n}}  \right)} \geq \prod\limits_{|z_{m,n}|\leq 2|z|} \frac{\varphi\left(z_{\rm min}  \right)}{\varphi\left( z_{\rm max}  \right)}. $$

	For the case $|z_{m,n}|>2|z|$ we use the same idea since $K_\varphi$ is non-zero and monotonic to estimate that
	$$\prod\limits_{|z_{m,n}|>2|z|} \frac{\varphi\left( \frac{z}{z_{m,n}}\right)}{\varphi\left( \frac{z}{\lambda_{m,n}} \right)} \geq \prod\limits_{|z_{m,n}|> 2|z|} \frac{\varphi\left(z'_{\rm min}  \right)}{\varphi\left( z'_{\rm max}  \right)},$$
where $z'_{\rm max}$ and $z'_{\rm min}$ are the $z\in \{ z_{m,n} :z_{m,n}>|z| \}$ that give the maximum and minimum of $\varphi(z)$, respectively.
	
	Therefore, we have the bound for $|h_2(z)| $.
	The estimate of $h_1(z)$ is given in~\cite{K292} [Page 4, term with $|h_2(z)|$], and since it is independent of the choice of $\varphi$ we obtain
	$$|h_2(z)|\geq Ce^{-c|z|\log |z|}.$$

	We note that, that if the growth of $\varphi$ is slower than the growth of the exponential, we simply use the estimate of $h_1(z)$, which explains the function $\gamma(z)$ from \eqref{eq:gamma} in the inequality \eqref{eq:inqaulity}.
	
From this we get now
$$|K_\varphi(-|z|^2)g(z) | = \left| K_\varphi(-|z|^2)\frac{ \sigma(z)} {{\rm d}(z, \Lambda)} {\rm d}(z, \Gamma)  h(z) \right| \geq \left| K_\varphi(-|z|^2)\frac{ \sigma(z)} {{\rm d}(z, \Lambda)}   {\rm d}(z, \Gamma) \right| C \gamma(z) e^{-c|z| \log|z|}.$$
	With the estimate $|K_\varphi(-|z|^2) \sigma(z)| \geq \tilde C {\rm d}(z, \Lambda)$ we obtain the result. This completes the proof.
\end{proof}

\begin{example}
	Consider the trivial example with $K_\varphi(z)=e^z$ in which we are back in the case of \cite{K92} and have the inequality
	\begin{equation*}
		\label{}
		c_1  e^{-c|z|\log |z|} \cdot {\rm dist}(z,\Gamma) ~\leq~ |e^{-|z|^2} g(z)|~\leq~  c_2  e^{c|z|\log|z|},
	\end{equation*}
	in which $c_1,c_2,c_3$ depend on $q(\Gamma)$ and $Q$ as previously defined.
\end{example}

Using Lemma \ref{lemma:ineq} we get the following Lagrange-type interpolation formula:
\begin{lemma}\label{lemma:3.2}
	Let $\Gamma=\{z_{m,n}\}$ be a uniformly close to the square lattice
	$\Lambda$, 	of density $\beta_\varphi$, and let $g$ be our fractional $\sigma$-function associated to $\Gamma$. If $\alpha < \beta_\varphi$, then we have
	\begin{equation*}
		f(z) = \sum_{m,n}\frac{f(z_{mn})}{g'(z_{mn})}\frac{g(z)}{z-z_{mn}}, \qquad z_{mn} \in \Gamma.
	\end{equation*}	
with uniform convergence on compact sets for $f\in \cF^{\infty}$.
\end{lemma}

The proof follows the same lines as the proof of Lemma 3.1 in~\cite{K292}.

This allows us now to prove our main theorems in the case of quasi-periodicity.
\begin{proof}(for Theorem~\ref{th:sam_and_inter_2} and Theorems~\ref{th:sam_and_inter_1})

Beurling (\cite{Beurling}, p.356) showed that for a square lattice $\Lambda$ of density $\frac{\alpha}{\pi}$, it is possible to find a subset $R\subseteq\Gamma$ uniformly closed to the lattice with $d< \frac{\alpha}{\pi}\leq d_1$ which implies that one can replace $\Gamma$ with $R$ and always assume $\Gamma$ to be uniformly close to the square lattice of density $\beta_\varphi=D^{-}(\Gamma)$. The same holds in our case and, therefore, we can assume $\Gamma$ to be uniformly close to the square lattice of density $\beta_\varphi=D^{-}(\Gamma)$.

Consider now the translation operator $T_a$
\begin{equation}
	\label{eq:m}
	(T_a f)(z) := \sqrt{\frac{K_\varphi\left( -|z-a|^2 \right)}{K_\varphi\left( -|z|^2 \right)}} f(z-a).
\end{equation}

We note that this translation acts isometrically in $\cF_\varphi$, since taking $w=z-a$ we have
	\begin{align*}
		\int\limits_{\mathbb C} |T_a f (z) |^2 d\mu(z) &= \iint\limits_{\BR^2}  \left|\frac{K_\varphi\left( -|z-a|^2 \right)}{K_\varphi\left( -|z|^2 \right)}\right| |f(z-a)|^2d K_\varphi(-|z|^2)dxdy,\\
		&= \iint\limits_{\BR^2}  \left|\frac{K_\varphi\left( -|w|^2 \right)}{K_\varphi\left( -|w+a|^2 \right)}\right| |f(w)|^2d K_\varphi(-|w+a|^2)dxdy,\\
		&= \iint\limits_{\BR^2}  \left|K_\varphi\left( -|w|^2 \right)\right| |f(w)|^2 dxdy,\\
		&=\int\limits_\mathbb C |f(w) |^2d\mu(w).
	\end{align*}

That these translation operators are isometrically invariant implies that $\Gamma+z$ is
a sampling set  if and only if $\Gamma$ is a sampling set. The same holds true for interpolation.

In particular, we get
$$
\iint\limits_{\BR^2}  \left|K_\varphi\left( -|w|^2 \right)\right| |f(w)|^2 dxdy =\sum_{k,l}\iint\limits_{R}  \left|K_\varphi\left( -|w|^2 \right)\right| | T_{\lambda_{k,l}}f(w)|^2 dxdy
$$
with $R=\{z=x+iy: |x|<\frac{1}{2}\sqrt{1/\alpha}, |y|<\frac{1}{2}\sqrt{1/\alpha}\}$ and $\alpha<\beta_\varphi$.

In order to estimate the terms $| T_{\lambda_{k,l}}f(w)|^2$ on the right-hand side we use Lemma \ref{lemma:3.2} with $a=\lambda_{k,l}$. This gives us %using $a=\lambda_{kl}$
\begin{align}
	\label{eq:Mf}
	(T_a f)(w)&= \sum_{m,n}\frac{(T_a f)(z_{m,n}+a)}{g'_{a} (z_{m,n}) }\frac{g_a(w)}{w-z_{m,n}-a}.
\end{align}
where the function $g_a(z)=g(z;\Gamma+a)$ is related to the lattice $\Gamma+a$ instead of the function $g$ where we sum over the lattice $\Gamma$ (c.f. \eqref{eq:g}).

Cauchy-Schwarz inequality implies
$$
| T_{\lambda_{k,l}}f(w)|^2 \leq \left(\sum_{m,n}\frac{(|T_a f)(z_{m,n}+a)|^2K_\varphi(-|z_{m,n}|^2)}{K_\varphi(-|z_{m,n}|^2) |g'_{a} (z_{m,n})|^2}\right)\left(\sum_{m,n}\frac{|g_a(w)|^2}{|w-z_{m,n}-a|^2}\right).
$$
Using Lemma \ref{lemma:ineq} we obtain
$$
| T_{\lambda_{k,l}}f(w)|^2 \leq C_1\left(\sum_{m,n}(|T_a f)(z_{m,n}+a)|^2K_\varphi(-|z_{m,n}|^2)e^{-c|z_{m,n}|\ln |z_{m,n}|}\right)\left(\sum_{m,n}\frac{|g_a(w)|^2}{|w-z_{m,n}-a|^2}\right)
$$
which leads to
\begin{eqnarray*}
&& \iint\limits_{R} | T_{\lambda_{k,l}}f(w)|^2 |K_\varphi( -|w|^2)|  dxdy\\ & \leq &  C_1\left(\sum_{m,n}(|T_a f)(z_{m,n}+a)|^2K_\varphi(-|z_{m,n}|^2)e^{-c|z_{m,n}|\ln |z_{m,n}|}\right)\left(\sum_{m,n}\iint\limits_{R}\frac{|g_a(w)|^2}{|w-z_{m,n}-a|^2}K(-|w|^2)dxdy\right)
\end{eqnarray*}

Collecting everything together we get
$$ \iint_{\BC}|f(z)|^2 K_\varphi(-|z|^2)dxdy \leq C_2 \sum_{m,n} |f(z_{m,n})|^2K_\varphi (-|z_{m,n}|^2)<\infty,$$
to prove Theorem \ref{th:sam_and_inter_1} in which it is sufficient to verify the left hand side of the inequality in Equation \eqref{eq:sampling_set}.

One could prove Theorem \ref{th:sam_and_inter_3} in a similar manner.

To prove Theorem~\ref{th:sam_and_inter_2} we write $f$ as:
$$
f(z) = \sum_{m,n} a_{mn}\sqrt{\frac{K_\varphi\left( -|z-a|^2 \right)}{K_\varphi\left( -|z|^2 \right)}}\frac{g_{-z_{mn}}(z-z_{mn})}{z-z_{mn}},
$$
Then, the interpolation problem is solved directly from the above formula using the translation operator \eqref{eq:Mf} and Lemma \eqref{lemma:ineq} and this completes the proofs.
\end{proof}

\section{Proofs of main theorems for the subharmonic case}\label{sec:proofs}

Under the condition that there exists subharmonic $\phi$, such that $\phi(z)=-\ln(K_\varphi(-|z|^2))$ we can prove the result by showing the required condition as in \cite{Beurling1}.
Impose the condition of existence of a subharmonic function $\phi$ such that and monotone $K_\varphi(-|z|^2)$,
\[ e^{-\phi} = K_\varphi(-|z|^2),\]
or equivalently,
\[ \phi = -\ln (K_\varphi (-|z|^2). \]
We note that $\phi$ is subharmonic as it is bounded above by $\ln(z)$ for $K_\varphi(-|z|^2)>0$.
Using the notation $\partial_z K_\varphi(-|z|^2)= \bar{z}(\partial_x K_\varphi(x))\mid_{x=-|z|^2} $ and similarly for $\partial_{\bar{z}}$,
the Laplacian acting on $\phi$ gives
\begin{equation}
	\label{eq:lap}
	\begin{aligned}
		\Delta \phi(z)&=\frac{1}{4} \frac{\partial}{\partial \overline{z}}\frac{\partial}{\partial z}\left(-\ln \left(K_{\varphi}(-z \bar{z})\right)\right)\\
		&=\frac{1}{4}\left[ 
		\frac{K_{\varphi}(-z \bar{z}) \partial_{\bar{z}}K_{\varphi}(-z \bar{z})
			+ z\bar{z}\partial_{\bar{z}}K_{\varphi}(-z \bar{z})\partial_{z}K_{\varphi}(-z \bar{z})
			-z\bar{z}\partial_z\partial_{\bar{z}}K_{\varphi}(-z \bar{z})
		}
		{K^2_{\varphi}(-z \bar{z})} \right]\\
			&=\frac{1}{4}\left[ 
		\frac{K_{\varphi}(-|z|^2) \partial_{\bar{z}}K_{\varphi}(-|z|^2)
			+|z|^2\partial_{\bar{z}}K_{\varphi}(-|z|^2)\partial_{z}K_{\varphi}(-|z|^2)
			-|z|^2\partial_z\partial_{\bar{z}}K_{\varphi}(-|z|^2)
		}
		{K^2_{\varphi}(-|z|^2)} \right] \\
		&=\frac{1}{4}\left[ 
		\frac{ \partial_{\bar{z}}K_{\varphi}(-|z|^2)}{K_{\varphi}(-|z|^2)}
		+\frac{|z|^2\partial_{\bar{z}}K_{\varphi}(-|z|^2)\partial_{z}K_{\varphi}(-|z|^2)}{K^2_{\varphi}(-|z|^2)}
	-\frac{|z|^2\partial_z\partial_{\bar{z}}K_{\varphi}(-|z|^2)
		}
		{K^2_{\varphi}(-|z|^2)} \right] 
	\end{aligned}
\end{equation}
Our goal would be to provide an upper and lower bound for $\Delta \phi$.  
That is, we want $m$ and $M$ such that $0< m\leq\Delta (\phi(z))\leq M< +\infty.$
By assumption, $K_{\varphi}(-z \bar{z})$ has the form  $e^{-\phi} = K_\varphi(-|z|^2)$ and $K_\varphi(-|z|^2)$ is a weight function and so it is positive and bounded. 
Thus, for each $ K_\varphi(-|z|^2)$ we know we have $M'_\varphi,m'_\varphi$ such that $0<m'_\varphi <K_\varphi(-|z|^2)<M'_\varphi<\infty$.\\

To show that $\Delta \phi$ is bounded from above, we bound the last two terms of \eqref{eq:lap} under the condition that
\begin{equation}\label{eq:cond}
	\partial_{\bar{z}} \left(\frac{\partial_z K_{\varphi}(-|z|^2)}{K_\varphi(-|z|^2)}\right) 
	=
	\frac{K_{\varphi}(-|z|^2)\partial_{z\bar{z}} K_{\varphi}(-|z|^2) -\partial_{\bar{z}}K_{\varphi}(-|z|^2)\partial_{z}K_{\varphi}(-|z|^2) }{K_{\varphi(-|z|^2)}}
	 \sim \frac{1}{|z|^2}.
\end{equation}
Multiplying both sides by $|z|^2$, we get,
\begin{equation*}\label{eq:4321}
	|z|^2\frac{K_{\varphi}(-|z|^2)\partial_{z\bar{z}} K_{\varphi}(-|z|^2)}{K_{\varphi(-|z|^2)}}
	-|z|^2\frac{ \partial_{\bar{z}}K_{\varphi}(-|z|^2)\partial_{z}K_{\varphi}(-|z|^2) }{K_{\varphi(-|z|^2)}}
	\sim 1.
\end{equation*}
For large $z$, since $K_{\varphi}(-|z|^2)$ is bounded we can treat it as constant multiplication.
Hence 
\[
|z|^2\frac{\partial_{z\bar{z}} K_{\varphi}(-|z|^2)}{K_{\varphi(-|z|^2)}}
-|z|^2\frac{ \partial_{\bar{z}}K_{\varphi}(-|z|^2)\partial_{z}K_{\varphi}(-|z|^2) }{K_{\varphi(-|z|^2)}}
\sim 1
\]
Hence the last two terms of Equation \eqref{eq:lap} have constant asymptotic behavior.
By condition \eqref{eq:cond}, $\frac{ \partial_{\bar{z}}K_{\varphi}(-|z|^2)}{K_{\varphi}(-|z|^2)}$  is also bounded from above hence, we have $M$ such that $\Delta \phi <M<\infty$.\\

A lower bound follows directly from the restriction that $K_\varphi(-|z|^2)$ is monotone, then $\partial_z K_{\varphi}(-|z|^2)>0$ and $ \partial_{\bar{z}}K_{\varphi}(-|z|^2)>0$, so a minimum exists, and consequently we have $m$ such that $m\leq\Delta \phi.$

To verify that $\Delta \phi$ is bounded away from zero, we note $K_\varphi(-|z|^2)$ is monotone, then $\partial_z K_{\varphi}(-|z|^2)\geq c> 0$ for some constant $c$.
So we must only verify that the two last terms of \eqref{eq:lap} are non negative.
By dividing the last two terms of \eqref{eq:lap} by $|z|^2$, it would be equivalent to solve the differential inclusion of the form $-\partial_x u(x) +u^2(x)\geq 0$, where $u(x) = \partial_{z}K_{\varphi}(-|z|^2)$. 
Since the ode
$-\partial_x u(x) +u^2(x) =0$
has general solution of the form $u(x) = \frac{1}{C-x}>0$ for $x\neq C$.
Hence, we are bounded away from $0$, meaning there exists $m$ and $M$ such that $0< m\leq\Delta (\phi(z))\leq M< +\infty.$\\

Additionally, we would need to show that $\Delta \phi$ is also uniformly Lipschitz, but by considering $\psi$ defined by
\[
\psi(z)=\frac{1}{zR^2}\int_{D(z,R)}\phi(z)d\mu(z),
\]
instead of $\phi$ we can always assume $\phi$ to be uniformly Lipschitz.\\

Hence we meet the necessary conditions, and result follows from \cite[Theorem 1 and Theorem 2]{Beurling1}.

Such $\phi$ exists for each holomorphic $K_\varphi(-|z|^2)$, namly we can consider the few example. 

\begin{example}
	Following the classical case in which $K_\varphi(z)= e^z$, as in Example \ref{exponential}. We get 
	\[\phi(z) = -\ln(e^{-|z|^2}) =|z|^2.\]
\end{example}

\begin{example}
	Following Example \ref{ex:eab}, in which $K_\varphi(x) = e^{ -a(-x)^b}$ for $x>0$, where $b >0$, ${\rm Re}(a)>0.$ Then, we have $\varphi(z)=\sum_{n=0}^\infty \varphi_n z^n,$ where
		$$\varphi_n = \frac{b a^{\frac{n+1}{b}} }{\Gamma\left( \frac{n+1}{b} \right)}, \qquad n =0, 1, 2, \ldots$$
	Then $\phi$ is given by 
	\[ 
	\phi(z) = -\ln(e^{-a|z|^b})=a|z|^b.
	\] 	
\end{example}

\begin{example}
Following Example \ref{ex:eab}, in which $K_\varphi(z) =\frac{1}{1-z}$ for $|z|<1$, we can find a subharmonic functions
$$\phi(z)=-\ln\left(\frac{1}{1-z}\right),$$
for all $z$ but in the ball of radius $1$.
\end{example}

\section{Applications}

Gabor Frames first developed by D. Gabor in 1946 in the area of information theory ~\cite{Gabor}, and J. Von Neumann in quantum mechanics at the same time.
Gabor Frames are used to establish frames for function systems depending on continuous parameters in modern signal and image processing.
For Gabor systems, Gr\"ochenig and Lyubarskii established a construction method with Hermite functions as window functions, which allows to connect the Gabor system with an orthonormal system in the Fock space via the Bargmann transform.	
This allows to reduce the problem of  lattice constants for the frame parameters to the one of proving the  uniqueness of sets of entire functions in the Fock space.

An application would be to perform a similar investigation of Gabor frames and sufficient conditions to form a frame in $L^2(\mathbb{R})$. In a similar spirit to the one done by K. Gr\"ochenig and Y. Lyubarskii (see ~\cite{GL2007, GL2009}) based on Hermite functions $h_n$, but in our framework of generalized differentiation through the generalized Bargmann transform and generalized Weierstrass-$\sigma$ function defined  in Sections~\ref{sec:B} and~\ref{sec:proofs}, respectively.\\

Let us recall the necessary results in the following theorem.

\begin{theorem}\label{thm:gerald}
	For the modified Bargmann transform $\tilde{\mathcal{B}}: L^2(\mathbb{R}) \to \mathcal{F}_\varphi$ we have:
	\begin{enumerate}[(i)]
		\item $\tilde{\mathcal{B}}$ is a unitary mapping.
		\item $\cF_\varphi$ is a RKHS with reproducing kernel $k(z,w)=\varphi(z\bar{w})$ and $F(z)=\iinner{ F , k(z,\cdot)}_{\cF,\varphi}$.
		\item Let $h_n(t)=e^{\pi t^2} \frac{d^n}{dt^n} (e^{-2\pi t^2 })$ denote the $n-$th Hermite function, then $\tilde{\mathcal{B}} h_n(z) = \pi^{\frac{n}{2}} ({\varphi_n})^{-\frac{1}{2}} z^n$.
	\end{enumerate}
\end{theorem}

We remark that
$(i)$ and $(iv)$ are obtained in Section \ref{sec:B}.
$(ii)$ is obtained in Section \ref{subSec:IP}, in particular, Equation \eqref{eq:ContinuousKernel}.

Based on this theorem we can now consider the construction of frames for integral transforms of the type
$$
V f(z) = \la k(z, \cdot), f\ra_{L^2(\mathbb{R})}
$$
with $f\in L^2(\mathbb R)$ and kernels of the form
$$
k(z, t)=\sum_{j=0}^\infty c_j(\overline{z}a^*+z a)^jh_0(t)
$$
depending on a complex parameter $z$, with coefficients $c_j$ such that the series converges in $L^2(\mathbb{R})$ with respect to $x$ and converges uniformly on compact sets with respect to $z$. Let us remark that a special case of such transforms is the Gabor transform:
\begin{equation}
		\label{eq:SFT}
		V_{h_n} f(z)
		= \la \pi_z h_n, f \ra_{L^2(\mathbb{R})}
		= \frac{ e^{i\pi x y } }{\sqrt{\pi^n \varphi_n }}\; \frac{1}{\varphi\left({\frac{|z|^2}{2}}\right)}  \sum_{k=0}^{n} {n \choose k} (-\pi \bar{z} )^k D_\varphi^{k} F(z)
	\end{equation}
with $z=x+iy$, which we get in the case where the operators $a$ and $a^*$ generate the Heisenberg algebra, that is $[a, a^\ast] = I$.

Using our generalized Bargmann transform $\tilde{\mathcal{B}}: L^2(\mathbb{R}) \to \mathcal{F}_\varphi$ as $F= \tilde{\mathcal{B}}f$ (c.f. Section~\ref{sec:B}) we have
\begin{align*}
	Vf(z) & = \la k(z,\cdot), f\ra_{L^2(\mathbb{R})}\\
	& = \iinner{ \tilde{\mathcal{B}}k(z,\cdot), F }_{\cF,\varphi}\\
	& = \iinner{ \sum_{j=0}^\infty c_j (\overline{z}M_w+z D_\varphi)^j 1, F }_{\cF,\varphi}\\
	& = \iinner{ \varphi(\overline{z} w), \frac{1}{\varphi(\overline{z} w)}\overline{ \left( \sum_{j=0}^\infty c_j (\overline{z}M_w+z D_\varphi)^j 1 \right) } F }_{\cF,\varphi}\\
    & = \frac{1}{\varphi(|z|^2)}\overline{\left(\sum_{j=0}^\infty c_j (\overline{z}M_w+z D_\varphi)^j 1 \right)(z)} F(z).\\
	\end{align*}

For an invertible matrix $C$ and a lattice $\Lambda = C\BZ^2$, we define $\displaystyle \cG(k,\Lambda) = \{ k(z,\cdot) : z  \in \Lambda   \}.$ $\cG(k,\Lambda) $ is a frame  if there exists constants $A,B$ with $0<A\leq B<\infty$ such that
$$ A\| f \|^2_{L^2(\mathbb{R})} \leq \sum_{z \in\Lambda} |\la f, k(z, \cdot\ra_{L^2(\mathbb{R}^2)} |^2 \leq B \|f \|^2_{L^2(\mathbb{R})},\quad \forall f\in{L^2(\mathbb{R})}.  $$
%for all $f\in{\ell^2_\varphi}$.

Here, we are going to adapt the classic scheme by Gr\"ochenig and Lyubarskii~\cite{GL2007, GL2009} to our setting. Of course, we recover the original case if we choose $\varphi (z)=e^{z}$ and $k(z,t)=\pi_\lambda h_n(t) = e^{ 2\pi i y t}h_n(t-x)$.

The following results will provide sufficient conditions for $\cG(k,\Lambda)$ to form a frame and describe its structure. Since in the previous sections we created the same machinery as in the case of Gabor frames with Hermite windows~\cite{GL2007} we can transfer the proofs to our case directly,
{for Gabor frame that $\cG(k,\Lambda)$ is a Bessel sequence.}%added this

%%%%%%%%%%%%%%%
\begin{theorem}%theorem 2.1
	If $\mathcal{G} (k,\Lambda)$ is a frame for $L^2(\mathbb{R})$ then the size of $\Lambda$ satisfies $s(\Lambda) = |{\rm det}(C)|\leq 1,$ where $s(\cdot)$ is the size of $\Lambda$ obtained from the sympletic area of a cell (see also \cite{Palamodov2005}).
\end{theorem}

The density of $\Lambda$ is denoted by $d(\Lambda) = s(\Lambda)^{-1}$ and equals to the Beurling density. The adjoint lattice is denoted by $\Lambda^{\circ} = s({\Lambda})^{-1} \Lambda$ with size $s(\Lambda^{\circ})=s(\Lambda)^{-1}$.

\begin{theorem}
	\label{th:2.2}
	Consider a lattice $\Lambda\subset \mathbb{R}^2$ with adjoint $\Lambda^0$. The following statements are equivalent
	\begin{enumerate}[(i)]
		\item $\mathcal{G}(k,\Lambda)$ is a frame for $L^2(\mathbb{R})$.
		\item There exists a function $\gamma\in L^2(\mathbb{R})$ such that $\mathcal{G}(\gamma,\Lambda)$ is a Bessel sequence in $L^2(\mathbb{R})$ and the inner product for all $\mu\in\Lambda^0$ satisfies
		$\la k(\mu,\cdot), \gamma \ra_{L^2(\mathbb{R})} = \delta_{\mu,0}$ for all $\mu\in\Lambda^{\circ}$.
	\end{enumerate}
\end{theorem}

The following theorem gives us a sufficient condition to form a frame in $L^2(\mathbb{R})$.

\begin{theorem}
	\label{th:3.1}
	Let $n\in \mathbb Z$, $n\geq 0$. If the size of $\Lambda$ satisfies $s(\Lambda) <(n+1)^{-1}$ then $\mathcal{G} (k,\Lambda)$ is a frame for $L^2(\mathbb{R})$.
\end{theorem}

The following is an example of these results with the Dunkl operator from Subsection~\ref{sec:dunkl}. It is based on the Fock space representation of the Dunkl-Gabor transform (see, e.g.~\cite{Mejj2012}). We remark that this representation is easy to obtain from~\cite{Folland, GL2007}, since in this case the group generated by the multiplication operator and the Dunkl operator is still the Heisenberg group. 
%%%%%%%%%%%%
\begin{example}
From Subsection~\ref{sec:dunkl}, we have that the Dunkl operator in the rank one case can be written as
$$ T_1 f(x) = \frac{\pr f}{\pr x } (x) + \kappa ~\frac{f(x)-f(-x)}{x},$$
associated to
$$\varphi(z) := e^z ~_1F_1(\kappa, 2\kappa+1; -2z),$$ with
$$\varphi_{2n}= \frac{\left(\frac{1}{2}\right)_n}{(2n)!\left(\kappa+\frac{1}{2}\right)_n} \qquad\mbox{and}\qquad \varphi_{2n+1}=\frac{\left(\frac{1}{2}\right)_{n+1}}{(2n+1)!\left(\kappa+\frac{1}{2}\right)_{n+1}},
$$ and where $(a)_0=1$ and $(a)_n = a(a+1) \cdots (a+n-1),$ for $a \in \mathbb{R} \setminus \mathbb{Z}^-_0.$

\begin{gather}
\varphi \left( |z|^2\right) = \sum_{j=0}^\infty \varphi_j  |z|^{2j}
= \sum_{k=0}^\infty \left[\varphi_{2k}  |z|^{4k} + \varphi_{2k+1} |z|^{4k+2} \right] \nonumber \\
= \sum_{k=0}^\infty \left[ \frac{ |z|^{4k} }{ (2k +1)(2k)! }   + \frac{   |z|^{4k+2} }{  (4k+6) (2k+1)! }  \right].
\end{gather}

This leads to the following representation of the Dunkl-Gabor transform
$$
		V_{h_n} f(z)
		= \la f, \pi_z h_n\ra
		= \frac{ e^{i\pi x y } }{\sqrt{\varphi_n }}\; \frac{1}{\varphi\left({\frac{|z|^2}{2}}\right)}  \sum_{k=0}^{n} {n \choose k} (-\pi \bar{z} )^k D_\varphi^k F(z)
$$
with $z=x+iy$.

By Theorem \ref{th:3.1}, choosing $z$ from a lattice $\Lambda$ of size less than $(n+1)^{-1}$ will give us a frame.

\end{example}

Another example is given in the following considerations.

\begin{example}
Le $A$ be an operator acting on the Fock space $\mathcal F.$ Then, this action is given by
\begin{gather*}
AF(z) := \iinner{\varphi(\ov{z} \cdot), AF }_{\mathcal F,\varphi} = \iinner{A^\ast_w \varphi(\ov{z} \cdot), F}_{\mathcal F, \varphi}
= \inner{\mathcal{B}^{-1} (A^\ast_w \varphi(\ov{z} \cdot)), f}_{L^2(\mathbb{R})} = \inner{k( \cdot ; z, A), f}_{L^2(\mathbb{R})},
\end{gather*} for all $F= \mathcal{B} f,$ where $f \in L^2(\mathbb{R}),$ and $k(t; z, A) := \mathcal{B}^{-1} (A^\ast_w \varphi(\ov{z} w))(t).$

Consider now $A$ of the form
$$A F(w) := \sum_{k=0}^n {n \choose k} (-\pi  w)^k D^k_{\varphi}F(w).$$
We obtain for its adjoint
$$A^\ast G(w) = \sum_{k=0}^n {n \choose k} (-\pi D_\varphi)^k [w^k G(w)]  = \sum_{k=0}^n {n \choose k} (-\pi)^k  D^{k}_\varphi [w^k G(w)].$$

When $G(w) = \varphi(\ov{z} w)$ we get
\begin{gather*}
D^{k}_\varphi \left( w^k \varphi(\ov{z} w) \right) = D^{k}_\varphi \left(  \sum_{j=0}^\infty \varphi_j  \ov{z}^j w^{j+k}  \right) = \sum_{j=0}^\infty \varphi_j  \ov{z}^j  D^{k}_\varphi \left(w^{j+k}  \right).
\end{gather*}
Computing these derivatives, we obtain
\begin{gather*}
D^{k}_\varphi \left(w^{j+k}  \right) =  \frac{\varphi_{j+k-1}}{\varphi_{j+k}} D^{k-1}_\varphi \left(w^{j+k-1}  \right) =  \frac{\varphi_{j+k-2}}{\varphi_{j+k}} D^{k-2}_\varphi \left(w^{j+k-2}  \right) = \ldots = \frac{\varphi_{j}}{\varphi_{j+k}} w^{j}.
\end{gather*} Hence, we get
$$A^\ast \varphi(\ov{z} w) = \sum_{k=0}^n {n \choose k} (-\pi)^k  \left( \sum_{j=0}^\infty \frac{\varphi_{j}^2} {\varphi_{j+k}}  (\ov{z}w)^j \right).$$
This means that we can consider the transform
$$
V_{h_n}f(z) := \inner{k(\cdot; z, A), f}_{L^2(\mathbb{R})}
$$
with kernel $k(t; z, A)=\sum_{k=0}^n {n \choose k} (-\pi)^k  \left( \sum_{j=0}^\infty \frac{\varphi_{j}^2} {\varphi_{j+k}}\ov{z}^j h_j(t) \right)$.\\

By Theorem~\ref{th:3.1} we have that a lattice $\Lambda$ gives us a frame if $s(\Lambda) <(n+1)^{-1}$ . By Theorem \ref{th:2.2}, we have now that there exists a $\gamma$ such that $\mathcal{G}(\gamma,\Lambda)$ is a Bessel sequence in $L^2(\mathbb{R})$ and the inner product satisfies $\la k(\cdot ; \mu, A ), \gamma \ra_{L^2(\mathbb{R})} = \delta_{\mu,0}$ for all $\mu$ in the adjoint lattice.

\end{example}

%%%%%%%%%%%%%%%%%%%%%%%%%%%%%%%%%%%%%%%%%%%%%%
\
%%%%%%%%%%%%%%%%%%%%%%%%%%%%%%%%%%%%%%%%%%%%%%

%~\\

\textbf{Acknowledgement}:
The work of the second and third authors were supported by Portuguese funds through the CIDMA-Center for Research and Development in Mathematics and Applications, and the Portuguese Foundation for Science and Technology (“FCT–Funda\c c\~ao para a Ci\^encia e a Tecnologia”) within Project ref. UID/4106/2025.
%%%%%%%%%%%%%%%%%%%%%%%%%%%%%%%%%%%%%%%%%%%%%%%%%

\begin{small}

\end{small}

\end{document}